\documentclass{article}

\usepackage[utf8]{inputenc} % For \'o

\usepackage{amsmath}
\usepackage{amsthm}
\usepackage{amssymb}
\usepackage{cases}
\usepackage{mathrsfs} % mathscr 用
\usepackage{ascmac}
\usepackage{framed}
\usepackage{fancybox}
\usepackage{ulem} % 訂正線  \sout{}

\usepackage[dvipdfmx]{graphicx}
\usepackage[dvipdfmx]{xcolor}
\usepackage{tikz}

\def\rnum#1{\expandafter{\romannumeral #1}}
\def\Rnum#1{\uppercase\expandafter{\romannumeral #1}} 

\theoremstyle{plain}
\newtheorem{dfn}{Definition}[section]
\newtheorem{thm}{Theorem}[section]
\newtheorem{pro}{Proposition}[section]
\newtheorem{cor}{Corollary}[section]
\newtheorem{lem}{Lemma}[section]
\newtheorem{rem}{Remark}[section]

\newcommand{\Rn}{\mathbb{R^{\textit{n}}}}
\newcommand{\N}{\mathbb{N}}
\newcommand{\Z}{\mathbb{Z}}
\newcommand{\R}{\mathbb{R}}

\newcommand{\lan}{\langle}
\newcommand{\ran}{\rangle}

\newcommand{\Sp}{\mathcal{S}}

\newcommand{\D}{\mathscr{D}}

\newcommand{\tkappa}{\tilde{\kappa}}
\newcommand{\hphi}{\hat{\phi}}
\newcommand{\sh}{\mathscr{S}}
\begin{document}
\title{A sharp sparse domination of pseudodifferential operators  }
\author{Ryosuke  Yamamoto\footnote{Department of Mathematical Sciences, Nagoya University, Aichi, 464-8601, Japan.
\newline e-mail: \texttt{d20004k@math.nagoya-u.ac.jp}. 
\newline 
\newline
\newline {\bf Keywards} sparse domination, Pseudodifferential operator.
}}   

\date{}
\maketitle

\begin{abstract}
In this paper, we give a sharp sparse domination of pseudodifferential operators associated with  symbols belonging to the H\"{o}rmander class, and fundamental solutions of dispersive equations. Furthermore, we give boundedness results of these operators on weighted Besov spaces by using the sparse domination.
\end{abstract}

\section{Introduction and results}
For any $m \in \mathbb{R}$ and $0\le \rho , \delta \le 1$, the H\"{o}rmander class $S^{m}_{\rho , \delta}$ is defined as the set of all $a\in C^{\infty}({\mathbb{R}}^{2n})$ such that
\begin{eqnarray*}
|{\partial}^{\beta}_{x}{\partial}^{\alpha}_{\xi} a(x,\xi)| \lesssim {(1 + |\xi|)}^{m-\rho |\alpha| + \delta |\beta|}
\end{eqnarray*}
for any $(x,\xi) \in {\mathbb{R}}^{2n}$. Here, $A \lesssim B$ means $A \le CB$ with a positive constant $C>0$. For given $a \in S^{m}_{\rho,\delta}$, we define the pseudodifferential operator $a(x,D)$ by
\begin{eqnarray*}
\displaystyle a(x,D)f(x)= \frac{1}{(2 \pi )^{n}}\int_{{\mathbb{R}}^{n}} e^{i x \xi } a(x, \xi ) \hat{f} (\xi )d{\xi} ,
\end{eqnarray*}
where $f \in \sh $ and $\hat{f}$ denotes the Fourier transform of $f$. Pseudodifferential operator is a useful tool for study of partial differential equations, and the many boundedness results are known. The most basic result is the $L^p$- boundedness given by H\"{o}rmander \cite{H-1} and Fefferman \cite{F}. H\"{o}rmander \cite{H-1} showed that $m \le -n(1-\rho)|1/2 -1/p|$ is necessary for $a(x,D)$ with $a \in S^{m}_{\rho ,\delta}$ to be $L^p$- bounded. Conversely, the $L^p$-boundedness of $a(x,D)$ with $a \in S^{m}_{\rho,\delta}$ and $m = -n(1-\rho)|1/2 -1/p|$ was established by Fefferman \cite{F}. As for the boundedness on Lebesgue spaces weighted by $\omega \in A_p$ which is so called Muckenhoupt weight, Miller \cite{M} established the $L^p(\omega)$-boundedness of $a(x,D)$ with $ a \in S^{0}_{1,0}$. For general $a \in S^{m}_{\rho,\delta}$, Michalowski, Rule and Staubach \cite{M-R-S-2} showed the $L^p(\omega)$-boundedness of $a(x,D)$ with $a \in S^{-n(1-\rho)}_{\rho , \delta}$ and $ \omega \in A_p$. Chanillo and Torchinsky \cite{C-T} showed it for a larger class $a \in S^{-n(1-\rho)/2}_{\rho,\delta}$ ($0\le \delta< \rho \le 1  $) and a smaller class $\omega \in A_{p/2}$, and Michalowski, Rule and Staubach \cite{M-R-S-1} showed the same result for $0< \delta= \rho <1$. It should be mentioned here that Beltran \cite{B} showed it for $a\in S^{m}_{\rho, \rho}$ with $-n(1-\rho)/2<m<-n(1-\rho)|1/2-1/p|$ and $\omega \in  A_{p/2} \cap {RH}_{(2t'/p)'} $, where $2 \le p <2t'$ and $t'$ is the conjugate exponent of $t=-n(1-\rho)/(2m)$. We remark that there is no such $p$ that satisfies $2 \le p <2t'$ for the critical exponent $m=-n(1-\rho)|1/2-1/p|$. An important idea to deduce weighted estimates is to show pointwise estimates. For example, Chanillo and Torchinsky \cite{C-T} established pointwise estimate
\begin{eqnarray*}
| {(a(x,D)f)}^*(x)| \lesssim  M_2 f(x) 
\end{eqnarray*}
for $a\in S^{-n(1- \rho)/2}_{\rho,\delta}$ ($0\le \delta <\rho \le1$), where ${(a(x,D)f)}^*$ denotes the sharp maximal function of $a(x,D)f$.

Recently as a refinement of pointwise estimates, the theory of sparse domination of operators was developed by Lerner \cite{L}. For operators $T$ on function spaces, the sparse domination means the inequalities:
\begin{eqnarray*}
\ \ |Tf(x)| \lesssim {\Lambda}_{ \mathcal{S} , r}f(x)  \ \ {\textstyle{and}} \ \ | \langle{Tf,g}  \rangle|\lesssim {\Lambda}_{{\mathcal{S}},r,s'}(f,g) .
\end{eqnarray*}
In particular, we call the first one \textit{sparse bounds} and the second one \textit{sparse form bounds}. See below for the definition of ${\Lambda}_{ \mathcal{S} , r}$ and ${\Lambda}_{{\mathcal{S}},r,s'}$.
\begin{dfn}
Let $\eta \in (0,1)$. A collection $\mathcal{S}$ of cubes in ${\mathbb{R}}^{n}$ is $\eta$-sparse family if there are pairwise disjoint  subsets  ${\{ E_Q\}}_{Q \in \mathcal{S} }$ such that $ E_{Q} \subset Q $, and $|E_{Q}|>\eta |Q|$. 
\end{dfn} 
We often just say \textit{sparse} instead of \textit{$\eta$-sparse} whenever there is no confusion. For any cube $Q$ and $p \in [1,\infty) $, we define $ {\langle f \rangle }_{p,Q} := {|Q|}^{-\frac{1}{p}} {||f||}_{ L^p(Q)} $. For a sparse collection $\mathcal{S}$ and $ r,s \in [1 ,\infty)$ , the $(r,s)$-sparse form operator ${\Lambda}_{\mathcal{S} , r,s}$ and $r$-sparse operator $ {\Lambda}_{ \mathcal{S} ,r}$ are defined by
\begin{eqnarray*}
 {\Lambda}_{ \mathcal{S} ,r}f(x) :=\sum_{Q \in \mathcal{S}} {\langle f \rangle }_{r,Q} 1_Q (x) \ \ , \ \ {\Lambda}_{\mathcal{S} , r,s}(f,g) := \sum_{Q \in \mathcal{S}} |Q| {\langle f \rangle }_{r,Q}{\langle g \rangle }_{s,Q}  
\end{eqnarray*}
for all $ f,g \in L^{1}_{loc} $. If $r<p<s$, we have 
\begin{eqnarray*}
{\Lambda}_{\mathcal{S} , r,s'}(f,g) \lesssim {||f||}_{p} {||g||}_{p'}.
\end{eqnarray*}
This inequality is easily obtained from the $L^p$-boundedness of $r$-Hardy Littlewood maximal operator $M_r$ which is defined by $M_r f(x) =\sup_ {Q\ni x} {\langle f \rangle}_{r,Q}$. Furthermore, weighted inequality with Muckenhoupt weights is deduced from sparse domination. Bernicot, Frey and Petermichl \cite{B-F-P} showed
\begin{eqnarray*}
{\Lambda}_{\mathcal{S} , r,s'}(f,g) \lesssim {(  {[\omega]}_{ A_{p/r}} {[\omega]}_{ {RH}_{ (s/p)'}})}^{\alpha} {||f||}_{L^p(\omega)} {||g||}_{ L^{p'}( {\omega}^{1-p'})} ,
\end{eqnarray*}
where $\alpha = \max (\frac{1}{p-r} , \frac{s-1}{s-p})$, $ {[\omega]}_{A_q} = {\sup}_{Q}   {\langle \omega \rangle }_{1,Q} { \langle {\omega}^{1-q'} \rangle}^{q-1}_{1,Q}$ and ${[\omega]}_{ {RH}_q} = {\sup}_{Q}   {\langle \omega \rangle }^{-1}_{1,Q}  \langle {\omega \rangle}_{q,Q}$ for any $1<q< \infty$. From these observations, sparse domination is used to study the weighted boundedness of operators, and Lerner \cite{L} gave the simple proof of $A_2$ conjecture which means 
\begin{eqnarray*}
{||Tf||}_{L^2(\omega)} \lesssim { [ \omega]}_{A_2} {||f||}_{L^2(\omega)},
\end{eqnarray*}
where $T$ denotes the Calder\'{o}n-Zygmund operators. The $A_2$ conjecture was studied by many researchers. For example, Petermichl \cite{P-1}, \cite{P-2} solved the $A_2$ conjecture for Hilbert transform and Riesz transform, and Perez, Treil and Volberg \cite{P-T-V} gave
\begin{eqnarray*}
{||Tf||}_{L^2(\omega)} \lesssim { [ \omega]}_{A_2} \log(1+{ [\omega]}_{A_2} ) {||f||}_{L^2(\omega)}
\end{eqnarray*}
for general Calder\'{o}n-Zygmund operators. Finally, $A_2$ conjecture was completely solved by Hyt\"{o}nen \cite{Hy}. Lerner \cite{L} gave another proof by establishing
\begin{eqnarray*}
{||Tf||}_X \lesssim \sup_{\Sp}{||{\Lambda}_{ \mathcal{S} , 1}f||}_X   
\end{eqnarray*}
for any Banach function space $X$, and it was improved to the pointwise estimate
\begin{eqnarray*}
|Tf(x)| \lesssim {\Lambda}_{ \mathcal{S} , 1}f(x)
\end{eqnarray*}
by Lerner \cite{L-2016}, Lerner and Nazarov \cite{L-N}. There are also results of sparse domination with other operators. Sparse form bounds of rough singular integral operators and Bochner-Riesz multipliers were shown by Conde-Alonso, Culic, Plinio and Ou \cite{C-C-P-O}, and Lacey, Mena and Reguera \cite{L-M-R} respectively. 

Beltran and Cladek \cite{B-C} discussed the sparse domination of pseudodifferential operators with symbols in $S^{m}_{\rho,\delta}$, and they established 
\begin{eqnarray*}
 |a(x,D)f(x)| \lesssim {\Lambda}_{ \mathcal{S} , r}f(x)  ,
\end{eqnarray*}
with $a \in S^{-n(1-\rho)}_{\rho ,\delta}$ and $1<r<\infty$ which implies the weighted boundedness result of \cite{M-R-S-2}, that is the $L^p(\omega)$-boundedness with $\omega \in A_p$. We establish a pointwise estimate of $a(x,D)$ with larger class $a \in S^{-n(1-\rho)/2}_{\rho. \rho}$ than $S^{-n(1-\rho)}_{\rho ,\rho}$ by introducing another type of sparse bounds:
\begin{thm}
\label{sparse bound}
Let $a \in S^{m}_{\rho, \rho} $ with $0< \rho <1$ and $m \in \mathbb{R} $. Then, for any $f \in L^{\infty}_{c}$, there exist the collection of finitely sparse families ${\{ {\mathscr{S}}_j \}}_{j=1}$ such that
\begin{eqnarray*}
|a(x,D)f(x)| \lesssim \sum_j \sum_{Q \in {\mathscr{S}}_j } {{\langle f \rangle}}_{2,Q} \sum_{R \subset Q , R \in {\mathscr{S}}_j } 1_R(x)
\end{eqnarray*}
if and only if
\begin{eqnarray*}
m  \le -n(1- \rho) /2.
\end{eqnarray*}
\end{thm}

Then as a corollary, we recover  the weighted boundedness result which was showed by Michalowski, Rule and Staubach \cite{M-R-S-1}, that is the $L^p(\omega)$-boundedness with $a \in S^{-n(1-\rho)/2}_{\rho,\rho}$ and $\omega \in A_{p/2}$. Furthermore, as a benefit of our new sparse bounds, we have the boundedness of pseudodifferential operators and also the time evolution $e^{it{(-\Delta)}^{\alpha/2}  }$ with $0<\alpha \le 2$ of dispersive equations on weighted Besov spaces (Theorem \ref{sparseform1}, Theorem \ref{sparseform2}, Corollary \ref{pse Besov}, Theorem \ref{sparseform3}, Corollary \ref{weight sol}). We have also the following Coifman-Fefferman estimate for $a(x,D)$ by the same argument used in the proof of Theorem \ref{sparse bound}.
\begin{thm}
\label{Coifman-Fefferman}
Let $a \in S^{-n(1- \rho)/2}_{\rho, \rho} $ with $0< \rho <1$. Then, for any $\omega \in A_{\infty}$ and $0 <p<\infty$, we have
\begin{eqnarray*}
{||a(x,D)f||}_{L^p(\omega)} \lesssim {[\omega]}_{A_{\infty} } {||M_2f||}_{L^p(\omega)}.
\end{eqnarray*}
\end{thm}

 \ \ 

This paper is organized as follows. In the next section, we prove Thoorem1.1 and Theorem 1.2 by using Lerner and Nazarov's method. The Section 3 is devoted to establish a sparse form bounds and the boundedness on weighted Besov spaces for $a(x,D)$ and $e^{it{(-\Delta)}^{\alpha/2}  }$, Furthemore, we give some results about the sharpness of weighted boundedness of these operators.

%%%% section 2%%%%%%%%%%%%%%%%%%%%%%%%%%%%%%%%%%%%%%%%%%%%%%%%%%%%%%%%%%%%%%%%%%%%%%%%%%%%%%%%%%%%%%%%%%%%%%%%
 \section{Sparse bounds for pseudodifferential operators}
 
\subsection{ The pointwise estimate for pseudodifferential operators}

To establish Theorem 1.1, we use the following definition of dyadic lattice and  sparse decomposition of measurable functions given by Lerner and Nazarov \cite{L-N}. 
\begin{dfn}
A Dyadic lattice $\D$ in $\Rn$ is any collection of cubes such that
 
 (D-1) if $Q \in \D$, then each child of $Q$ is in $\D$,
 
 (D-2) every two cubes in $\D$ have a common ancestor in $\D$,
 
 (D-3) $\D$ is regular, i.e., for any compact set $K$ in $\Rn$, there exists $Q \in \D$ such that $K \subset Q$.

\end{dfn}

\begin{thm}[\cite{L-N}]
\label{L-N}
Let $f: \Rn \to \R$ be any measurable almost everywhere finite function such that for every $\varepsilon >0$,
\begin{eqnarray*}
\lim_{R \to \infty} {R}^{-n}|\{ x \in { [-R,R]}^n \ ; \ |f(x)| >\varepsilon \} =0 .
\end{eqnarray*}
Then, for any dyadic lattice $\D$ and any $\lambda \in (0, 2^{-n-2} ]$, there exists the sparse family $\Sp \subset \D $ such that
\begin{eqnarray*}
|f(x)|\le \sum_{Q \in \Sp} {\omega}_{\lambda} (f;Q)1_Q(x),
\end{eqnarray*}
where
\begin{eqnarray*}
{\omega}_{\lambda}(f;Q)= \inf_{\substack{E\subset Q \\ |E|>(1-\lambda )|Q|}} \sup_{x,x' \in E} |f(x)-f(x')|.
\end{eqnarray*}
\end{thm}

By using Theorem \ref{L-N}, we have a pointwise estimate of $a(x,D)$ with $a \in S^{-n(1 -\rho)/2}_{\rho, \rho}$:

\begin{lem}
\label{pointwise}
Let $a \in S^{-n(1-\rho)/2}_{\rho, \rho} $ with $0< \rho <1$. Then, for any $f \in L^{\infty}_{c}$, there exist the sparse family $\Sp$ so that
\begin{eqnarray*}
|a(x,D)f(x)| \lesssim \sum_{k\ge 0} 2^{-\varepsilon k}  \sum_{Q \in \Sp , |Q| \ge 3^{-\frac{2n}{1-\rho}}   } { \langle f \rangle}_{2,2^{k+1} Q} 1_Q (x) +  \sum_{k \ge 0} 2^{- \varepsilon k}  \sum_{Q \in \Sp , |Q|< 3^{-\frac{2n}{1-\rho}}    } { \langle f \rangle}_{2,2^{{k+1}}Q^{\rho}} 1_Q (x),
\end{eqnarray*}
where $\varepsilon = \lfloor n/2 \rfloor- n/2 +1$.
\end{lem}

To prove the lemma, we give a partition of unity. Take $\hat{\psi} \in C^{\infty}_{0}({\mathbb{R}}^n)$ such that supp $ \hat{\psi} $ $ \subset B(0,2)  $, $ \hat{\psi} =1$ on $B(0,1)$ and $ \psi \ge 0$, and denote $   { \hat{\psi} }_{j}(\xi) := {\hat{\psi}} (2^{-j} \xi) - \hat{\psi}(2^{-j+1} \xi)$ for $j \in \Z$, 
\[
\phi_j=\begin{cases}
\psi_j & j \in \N \\
{ {\sum_{ i \le 0 } \psi_i }}&  j=0
\end{cases}.
\]
Then, $a(x,D)$ is decomposed as
\begin{eqnarray*}
a(x,D)=\sum^{\infty}_{j=0} a_j(x,D),
\end{eqnarray*}
where $a(x,\xi)= a(x,\xi) {\hphi}_j (\xi)$. Furthermore, we use these notations in the following sections. Let us prove Lemma \ref{pointwise}.
\begin{proof}
From Theorem 2.1, we have
\begin{eqnarray*}
|a(x,D)f(x)| \le \sum_{Q \in \Sp } {\omega}_{\lambda}(|a(x,D)f(x)| ;Q)1_Q(x).
\end{eqnarray*}
First, we consider the case $|Q|<3^{-\frac{2n}{1-\rho}} $. Let $\alpha >0$ and
\begin{eqnarray*}
E=\{ x \in Q \ ; \ |a(x,D)(f1_{2Q^{\rho} })| \le \alpha \}.
\end{eqnarray*}
Then, $L^2 \to L^{2/\rho}$ boundness of $a(x,D)$ yields 
\begin{eqnarray*}
|E^c|^{\rho /2} &\le& {\alpha}^{-1} {||a(x,D)(f1_{2 Q^{\rho}}) ||}_{L^{2/\rho}} \\
&\le&  {\alpha}^{-1} {||a(x,D) ||}_{L^2 \to L^{2/\rho}} {||f||}_{L^2(2 Q^{\rho} )}.
\end{eqnarray*}
By taking $\alpha = 2^{n}{\lambda}^{-\rho/2} {||a(x,D) ||}_{L^2 \to L^{2/\rho}} {\langle f \rangle}_{2, 2 Q^{\rho} }$, one has $|E^c|\le \lambda|Q|$ and $|E| \ge (1-\lambda)|Q| $. Therefore, we have
\begin{eqnarray*}
|a(x,D)f(x) -a(x,D)f(x')| \lesssim   {\langle f \rangle}_{2, 2 Q^{\rho} } + | a(x,D)(f1_{ {(2 Q^{\rho}) }^c })(x) - a(x,D)(f1_{ {(2 Q^{\rho}) }^c })(x')|
\end{eqnarray*}
for any $x,x' \in E$. We estimate the second term. Let $a_j (x, \xi) := a(x, \xi) {\hphi}_j  (\xi)$ and
\begin{eqnarray*}
K_j(x,y) =\int e^{i(x-y) \xi} a_j(x, \xi) d\xi .
\end{eqnarray*}
We integrate by parts in $\xi$ to obtain
\begin{eqnarray*}
|K(x,y)| \lesssim {|x-y|}^{-N}  \sum_{|\alpha| =N} \left| \int e^{i(x-y)\xi} {\partial}^{\alpha}_{\xi} a_j(x,\xi) d\xi \right|
\end{eqnarray*}
for any $n \in \N$. Hence, we have
\begin{eqnarray*}
|a(x,D)(f1_{{(2 Q^{\rho} )}^c })(x)| &\le&  \sum_{|\alpha| =N}  \int {|x-y|}^{-N} |f(y)|1_{{(2 Q^{\rho} )}^c }(y) \left| \int  e^{i(x-y)\xi} {\partial}^{\alpha}_{\xi} a_j(x,\xi) d\xi \right| dy \\
&=& \sum_{|\alpha| =N}   \sup_{{||g||}_{L^{\infty}}=1 } \left| \int {|x-y|}^{-N} f(y)1_{{(2 Q^{\rho} )}^c }(y) g(y) \int  e^{i(x-y)\xi} {\partial}^{\alpha}_{\xi} a_j(x,\xi) d\xi dy \right| \\
& \le & \sum_{|\alpha| =N}   \sup_{{||g||}_{L^{\infty}}=1 }  {\left( \int {| {\partial}^{\alpha}_{\xi} a_j(x,\xi)|}^2 \right)}^{1/2} {||\mathcal{F}[ {|x- \cdot|}^{-N} f1_{{(2 Q^{\rho} )}^c }g||}_{L^2}\\
&\lesssim&2^{j \rho n/2 - j \rho N} {\left( \int_{ {(2 Q^{\rho} )}^c} {|x-y|}^{-2N}   {|f(y)|}^2 dy \right)}^{1/2} \\
&\lesssim&2^{j\rho n/2 - j \rho N} \sum_{k \ge 1}  {\left( \int_{2^{k+1} Q^{\rho}  \setminus 2^k  Q^{\rho} } {|x-y|}^{-2N}   {|f(y)|}^2 dy \right)}^{1/2} \\
&\lesssim&2^{j\rho n/2 - j \rho N} {\ell(Q)}^{-\rho N +\rho n/2 } \sum_{k \ge 1} 2^{-kN +kn/2} {\langle f \rangle}_{2,2^{k+1} Q^{\rho} } .
\end{eqnarray*}
By taking $N>n/2$, one has
\begin{eqnarray*}
\sum_{2^{-j} \le {\ell(Q)} } |a(x,D)(f1_{{(2 Q^{\rho} )}^c })(x)|  \lesssim \sum_{k \ge 1} 2^{-kN +kn/2} {\langle f \rangle}_{2,2^{k+1} Q^{\rho} } . 
\end{eqnarray*}
On the other hands, it holds that
\begin{eqnarray*}
&&{(x-y)}^{\alpha} \{  K_j(x,y) - K_j(x',y) \}  \\
&=&(x-y)^{\alpha} \int e^{i(x-y)\xi}(1- e^{-i(x-x')\xi})a_j(x,\xi)d \xi + (x-y)^{\alpha} \int e^{i(x'-y)\xi}(a_j(x,\xi) -a_j(x', \xi) )  d \xi \\
&=& \int e^{i(x-y)\xi}{\partial}^{\alpha}_{\xi} \{ (1- e^{-i(x-x')\xi})a_j(x,\xi)\} d \xi +  \int e^{i(x'-y)\xi} {\partial}^{\alpha}_{\xi} (a_j(x,\xi) -a_j(x', \xi) )  d \xi .
\end{eqnarray*}
For any $j$ such that $2^{-j} >\ell(Q)$, Taylor's formula yields
\begin{eqnarray*}
|{\partial}^{\alpha}_{\xi} \{ (1- e^{-i(x-x')\xi})a_j(x,\xi)\}| \lesssim  \ell(Q) 2^{-jn(1-\rho)/2+j -j \rho |\alpha| },
\end{eqnarray*}
and
\begin{eqnarray*}
|{\partial}^{\alpha}_{\xi} (a_j(x,\xi) -a_j(x', \xi) )| &=& \left| {\partial}^{\alpha}_{\xi}\int^{1}_{0} (x-x') \cdot (\nabla_{x} a_j)(x'+t(x-x'), \xi) dt \right| \\
&\lesssim& \ell(Q) 2^{-j n (1-\rho)/2-j\rho|\alpha| + j\rho}.
\end{eqnarray*}
From these results, we obtain
\begin{eqnarray*}
&&\sum_{2^{-j} > \ell(Q)} | a_j(x,D)(f1_{ {(2 Q^{\rho}) }^c })(x) - a_j(x,D)(f1_{ {(2 Q^{\rho}) }^c })(x')|\\
&\lesssim& \sum_{2^{-j} > \ell(Q)} 2^{j\rho n/2 +j - j \rho N} {\ell(Q)}^{1-\rho N +\rho n/2 } \sum_{k \ge 1} 2^{-kN +kn/2} {\langle f \rangle}_{2,2^{k+1} Q^{\rho} } \\
&\lesssim&  \sum_{k \ge 1} 2^{-kN +kn/2} {\langle f \rangle}_{2,2^{k+1} Q^{\rho} }
\end{eqnarray*}
by taking $N=\lfloor n/2 \rfloor +1$. In the case $|Q| \ge3^{-\frac{2n}{1-\rho}}    $, the desired estimate is easily checked in the same way as above by setting
\begin{eqnarray*}
E=\{ x \in Q \ ; \ |a(x,D)(f1_{2Q })| \le \alpha \} \ , \ \alpha = 2^{n}{\lambda}^{-1/2} {||a(x,D) ||}_{L^2 \to L^{2}} {\langle f \rangle}_{2, 2 Q }.
\end{eqnarray*}
\end{proof}

%%%% section 3%%%%%%%%%%%%%%%%%%%%%%%%%%%%%%%%%%%%%%%%%%%%%%%%%%%%%%%%%%%%%%%%%%%%%%%%%%%%%%%%%55
\subsection{ Proof of Theorem 1.1}

In this subsection, we prove Theorem \ref{sparse bound} by using Lemma 2.1 and Lerner and Nazarov's technique \cite{L-N}. 

\begin{dfn}
Let $\mathcal{P}$ denotes a map from $\{ (Q,Q') \in \mathscr{D} \times \mathscr{D} \ ; \ Q' \subset Q \}$ to $\{ true,false\}$ such that $\mathcal{P}(Q,Q)=true$ for any $Q\in \mathcal{D}$. Then, we call that $(Q,Q')$ is one step if ${\mathcal{P}}(Q,Q')=false$ and $\mathcal{P}(Q,R)=true$ for any $Q \subsetneq R \subset Q$, and we call that $(Q,Q')$ is finite step if there exist
 $m \in \mathbb{N}$ and sequence $Q'=Q_0 \subset Q_1 \subset \cdots \subset Q_m =Q$ such that each $(Q_{j+1},Q_j)$ is one step.  Furthermore, we set
 \begin{eqnarray*}
 stop(Q, \mathcal{P} )=\{ Q' \in \mathscr{D} \ ; \ (Q,Q') \ is \ finite \ step \}.
 \end{eqnarray*}
\end{dfn}

Let $\Sp \subset \D $ denotes a sparse family and that with every cube $Q \in \Sp$ some family $\mathscr{F}(Q) \subset \D$ of child of $Q$ is associated so that $Q \in \mathscr{F} (Q)$.  Then, we define the family of cubes $\tilde{\Sp}$ by 
\begin{eqnarray*}
\tilde{\Sp} := \bigcup_{Q \in \Sp} \tilde{\mathscr{F}} (Q) , \ \  \ \ \ \ \ \ \ \ \ \ \ \ \ \ \ \ \ \ \ \ \ \ \ \ \  \ \\
\tilde{ \mathscr{F}}(Q) := \{ P \in \mathscr{F} (Q) \ ; \ P \notin \mathscr{F}(R) \ for \ any \ Q \subsetneq R \},
\end{eqnarray*}
and we call the argumentation of $\Sp$ by $\mathscr{F}(Q)$. In \cite{L-N}, Lerner and Nazalov proved $\tilde{\Sp}$ be a sparse family if $\mathscr{F}(Q)$ are sparse families. In particular, they proved the following result in the same paper.

\begin{pro}
Let $\Sp$ be a sparse family  and assume that 
\begin{eqnarray*}
\sum_{j} |Q_j| <\frac{1}{2} |Q|
\end{eqnarray*}
for any $Q \in \Sp $ and finitely pairwise disjoint cubes ${\{ Q_j \}}_j$ included $Q$ such that $ \mathcal{P} (Q,Q_j) =false$. Then, the augmentation of $\Sp$ by $stop(Q,\mathcal{P} )$ is a sparse family. 
\end{pro}

Let us prove Theorem 1.1.
\begin{proof}
In view of the three lattice theorem in \cite{L-N}, there exists the family of dyadic lattices ${\{{ \mathscr{D}}_j \}}_{j=1,2, \cdots ,3^{2n} }$ so that $Q^{\rho} \subset R_Q \in {\D}_{j}$, $2^k Q^{\rho} \subset R \in {\D}_j$ and $|Q^{\rho} | \sim |R_Q|$, $|2^k {Q}^{\rho} | \sim |R|$ with  some $j$. From this, we have
\begin{eqnarray*}
 \sum_{k \ge 0} 2^{- \varepsilon k}  \sum_{\substack{Q \in \Sp \\ |Q| < 3^{-\frac{2n}{1-\rho}}  } } { \langle f \rangle}_{2,2^{{k+1}}Q^{\rho}} 1_Q (x) \lesssim \sum_j \sum_{\substack{Q \in {\Sp} \\ |Q|< 3^{-\frac{2n}{1-\rho}}   } } \sum_{\substack{R \in {\D}_j \\ R_Q \subset R} } { \left( \frac{|R_Q|}{|R|} \right)}^{\varepsilon} { \langle f \rangle}_{2,R} 1_Q (x) .
\end{eqnarray*}
Furthermore, we take $\overline{Q} \in {\D}_j$ such that $Q \subset \overline{Q}$ and $|\overline{Q}|=3^{2n}|Q|$ for any $j$, and set ${\Sp}_j=\{ \overline{Q} \ ; \ Q \in \Sp \} $, ${{\Sp}}_j'= \{ \overline{Q} \ ; \ Q \in \Sp ,  \ |Q| <3^{-\frac{2n}{1-\rho}}    \}$, of course $\Sp_j$ be a regular sparse collection. Since $Q \to \overline{Q}$ is a injective map,  we can define the $R_{\overline{Q} } :=R_{Q}$. Here, the assumption $|Q|<3^{-\frac{2n}{1-\rho}}$ gives
\begin{eqnarray*}
|\overline{Q}| =3^{2n}|Q| < {|Q|}^{\rho} \le |R_Q| ,
\end{eqnarray*}
which yields $\overline{Q} \subset R_{\overline{Q}}$. From these results, for any regular sparse family $\overline{{\Sp}_j }$ so that $ \Sp_j \subset \overline{{\Sp}_j }  \subset {\D}_j $, we obtain
\begin{eqnarray*} 
\sum_j \sum_{\substack{ Q \in {\Sp} \\ |Q| <  3^{-\frac{2n}{1-\rho}}      } } \sum_{\substack{R \in {\D}_j \\ R_Q \subset R }} { \left( \frac{|R_Q|}{|R|} \right)}^{\varepsilon} { \langle f \rangle}_{2,R} 1_Q (x) &\lesssim& \sum_j \sum_{Q \in \Sp_j'  } \sum_{\substack{R \in {\D}_j \\ R_Q \subset R} } { \left( \frac{|R_Q|}{|R|} \right)}^{\varepsilon} { \langle f \rangle}_{2,R} 1_Q (x) \\
&=& \sum_j  \sum_{U \in \overline{{\Sp}_j }} \sum_{Q \in {\Sp_j'} } \sum_{ \substack{ R \in H_{\overline{{\Sp}_j }}(U) \\ R_Q \subset R} }{ \left( \frac{|R_Q|}{|R|} \right)}^{\varepsilon} { \langle f \rangle}_{2,R} 1_Q (x) \\
&\le& \sum_j \sum_{U \in \overline{{\Sp}_j }} \sup_{R \in H_{  \overline{{\Sp}_j } } (U)} {\langle f \rangle}_{2,R}  \sum_{\substack{Q \in \Sp_j' \\ Q \subset U }} \sum_{ \substack{ R \in H_{\overline{{\Sp}_j }  }(U) \\ R_Q \subset R} } { \left( \frac{|R_Q|}{|R|} \right)}^{\varepsilon}  1_Q (x) ,
\end{eqnarray*}
where
\begin{eqnarray*}
H_{\overline{{\Sp}_j }}(U) := \{ R \in \D_j \ ; \ R \subset U, \  there \ is \ no \ cube \ P\in  \overline{\Sp_j} \ so \ that \ R \subsetneq P \subsetneq U \}.
\end{eqnarray*}
Since
\begin{eqnarray*}
\sum_{ \substack{ R \in H_{\overline{{\Sp}_j }}(U) \\ R_Q \subset R} } { \left( \frac{|R_Q|}{|R|} \right)}^{\varepsilon} \lesssim 1 ,
\end{eqnarray*}
one has
\begin{eqnarray*}
 \sum_{k \ge 0} 2^{- \varepsilon k}  \sum_{\substack{Q \in \Sp \\ |Q| < 3^{-\frac{2n}{1-\rho}}  } } { \langle f \rangle}_{2,2^{{k+1}}Q^{\rho}} 1_Q (x) \lesssim \sum_j \sum_{U \in {\overline{\Sp_j}}} \sup_{R \in H_{ \overline{{\Sp}_j }} (U)} {\langle f \rangle}_{2,R}  \sum_{\substack{Q \in \overline{{\Sp}_j  }\\ Q \subset U }}  1_Q (x) .
\end{eqnarray*}
If
\begin{eqnarray*}
\sup_{R \in H_{ \overline{{\Sp}_j }} (U)} {\langle f \rangle}_{2,R} \lesssim {\lan f\ran}_{2,U}
\end{eqnarray*}
holds, the proof will be completed. We define the map $\mathcal{P}$ by
\begin{eqnarray*}
\mathcal{P}(U,R)=\left\{  \begin{array}{ll} 
true &  {\langle f \rangle}_{2,R} \le \sqrt{2} {\langle f \rangle}_{2,U}\\
false & other
\end{array} \right. .
\end{eqnarray*}
Let ${\{R_j \}}_j$ be a pairwise disjoint dyadic child of $U$ such that $\mathcal{P} (U,R_j)= false$, then we have
\begin{eqnarray*}
\sum_j |R_j| \le \frac{1}{2} |U| \sum_j {||f||}^{2}_{L^2(R_j)} {||f||}^{-2}_{L^2(U)} \le \frac{1}{2} |U|.
\end{eqnarray*}
Hence, the argumentation of ${\Sp}_j$ by $stop(U,\mathcal{P})$ be a regular sparse family and set $\overline{\Sp_j}$. We assume that there exists the $ R \in H_{ \overline{\Sp_j} }(U)$ so that $\mathcal{P}(U,R) =false$. From the definition of $H_{ \overline{\Sp_j} }(U)$, we obtain $R \notin \overline{\Sp_j}$ which yields $R \notin stop(U, \mathcal{P})$. We take $R \subsetneq R_1 \subset U$ such that $\mathcal{P}(U,R_1)=false$. If $R_1 \neq U$,  we can take $R_1\subsetneq R_2 \subset U$ such that $\mathcal{P}(U,R_2)=false$ again. By repeating this work, we have $\mathcal{P}(U,U)= false$ which contradict the definition of $\mathcal{P}$. Hence, we have $\mathcal{P}(Q,R)=true$ and
\begin{eqnarray*}
\sup_{R\in H_{ {\Sp}_0}(Q) } {\langle f \rangle}_{2,R} \lesssim {\langle f \rangle}_{2,Q}.
\end{eqnarray*}
\end{proof}
%%%%%%%%%%%%%%%%%%%%%%%%%%%%%%%%%%%%%%%%%%%%%%%%%%%%%%%%%%%%%%%%%%%%%%%%%%%%%%%%%%%%%%%%%%%%%%%%%%%%%%%%
\subsection{Weighted $L^p$ bounds for pseudodifferential operators}

This subsection is devoted to prove Theorem \ref{Coifman-Fefferman}. The class $ A_{\infty}$ denotes the set of all nonnegative locally integrable function $\omega$ such that

\begin{eqnarray*}
{[\omega]}_{A_{\infty}} := \sup_{Q} \frac{1}{\omega(Q)} \int_{Q} M(\omega 1_Q) < \infty.
\end{eqnarray*}
The sharp reverse H\"{o}lder inequality of $A_{\infty}$ weights was shown by Hyt\"{o}nen and P\'{e}rez \cite{H-P} .

\begin{thm}[\cite{H-P}]
\label{A infty}
Let $\omega \in A_{\infty}$. Then, there exists a constant $c_n$ depends on dimension $n$ such that 
\begin{eqnarray*}
{\left( \frac{1}{ |Q| } \int_Q {\omega}^{ \delta} \right)}^{ {1} / {\delta} } \le \frac{2}{|Q|} \omega(Q)
\end{eqnarray*}
 for any $Q$ where $\delta=1+c_n {[\omega]}^{-1}_{A_{\infty} }$.
\end{thm}

From this theorem, we remark that
\begin{eqnarray*}
\int_Q |f| \omega &\le& { \left( \int_Q {|f|}^{ {\delta}'} \right)}^{{1} / { {\delta}'} } {\left( \int_Q {\omega}^{\delta} \right)}^{{1} / {\delta}} \\
&\le&\frac{2}{|Q|} \omega(Q) { \left( \int_Q {|f|}^{ {\delta}'} \right)}^{{1}/{ {\delta}'} } 
\end{eqnarray*}
for each nonnegative locally integrable function $f$. In particular, for any measurable subset $E \subset Q$, we have
\begin{eqnarray*}
\omega(E) \le 2 {\left( \frac{|E|}{|Q|} \right)}^{1/{\delta}'} \omega(Q)
\end{eqnarray*}
by taking $f=1_E$. To establish Theorem \ref{Coifman-Fefferman}, it suffices to prove following estimate which is shown by using Cejas, Li, P\'{e}rez and Rivera-R\'{ı}os's idea in \cite{C-L-P-R}.

\begin{lem}
Let $X : \{ cube \} \to  \{cube\}$ be a map such that $Q \subset X(Q)$ for any cube $Q$. and let
\begin{eqnarray*}
{\Lambda}_{\mathscr{S} ,r ,X} f(x) := \sum_{Q \in \Sp} { \lan f \ran}_{r,X(Q)} 1_Q(x)
\end{eqnarray*}
for any sparse family $\Sp$ and $1 \le r <\infty$. Then, for any $\omega \in A_{\infty} $ and $p \in (0, \infty)$, one has
\begin{eqnarray*}
{||{\Lambda}_{\Sp,r ,X} f ||}_{L^p (\omega)} \lesssim {[\omega]}_{A_{\infty} }{||M_rf||}_{L^p(\omega)}
\end{eqnarray*}
for any $ f \in L^{\infty}_{c}$.
\end{lem}

\begin{proof}
Let $\gamma >0$ and we have
\begin{eqnarray*}
{||{\Lambda}_{\Sp,r ,X} f ||}^{p}_{L^p (\omega)} &\lesssim& \sum_{k \in  \Z} 2^{kp} \omega(\{    {\Lambda}_{\Sp,r ,X} f >2^k \} )\\
&\le&  \sum_{k \in \Z} 2^{kp} \omega(\{ {\Lambda}_{\Sp,r ,X} f >2^k ,  \ M_r f \le \gamma 2^k\} ) + \sum_{k\in \Z} 2^{kp} \omega(\{ M_rf > \gamma 2^k \}) \\
&\lesssim&  \sum_{k \in \Z} 2^{kp} \omega(\{ {\Lambda}_{\Sp,r ,X} f >2^k ,  \ M_r f \le \gamma 2^k\} ) + {\gamma}^{-p} {||M_rf||}^{p}_{L^p(\omega)} .\\
\end{eqnarray*}
Here, we set
\begin{eqnarray*}
{\Sp}_m &=& \{ Q \in \Sp \ ; \ 2^m \le {\lan f \ran}_{r ,X(Q)} < 2^{m+1} \}, \\
{\Sp}^{*}_m &=& \{ Q \in {\Sp}_m  \ ; \ Q  \ is \  maximal \ with \ inclusion \}
\end{eqnarray*}
for any $m \in \Z$. If $2^m>\gamma 2^k$, we obtain $M_rf(x) > \gamma 2^k$ for any $x \in Q \in {\Sp}_{m}$ from the assumption $Q \subset X(Q)$. Hence, one obtains
\begin{eqnarray*}
\omega(\{ {\Lambda}_{\Sp,r ,X} f >2^k , M_r f \le 2^k\} ) &=& \omega \left( \left\{ \sum_{2^m \le \gamma 2^k} {\Lambda}_{{\Sp}_m ,r ,X} f>2^k,  \ M_rf \le \gamma 2^k \right\} \right) \\
&\le& \sum_{  2^m \le \gamma 2^k   }  \omega ( \{ {\Lambda}_{{\Sp}_m ,r ,X} f>{\gamma}^{-1/2}2^{(m+k)/2-1} \} ) \\
& \le& \sum_{2^m \le \gamma 2^k}  \omega \left( \left\{ \sum_{Q \in {\Sp}_m} 1_Q >{\gamma}^{-1/2} 2^{(-m+k)/2-2} \right\}  \right) \\
&\le & \sum_{ 2^m \le \gamma 2^k} \sum_{U \in {\Sp}^{*}_{m} } \omega \left( \left\{ x \in U \ ; \   \sum_{Q \in {\Sp}_m , Q \subset U} 1_Q(x) > {\gamma}^{-1/2} 2^{(-m+k)/2-2} \right\}  \right) \\
&=: &\sum_{ 2^m \le \gamma 2^k} \sum_{U \in {\Sp}^{*}_{m} } \omega (E).
\end{eqnarray*}
For any $s \in (1, \infty)$, the sparseness of ${\Sp}_m$ gives 
\begin{eqnarray*}
{\left| \left| \sum_{Q \in {\Sp}_m , Q \subset U} 1_Q \right| \right| }_{L^s} &\le& \sup_{ {||g||}_{L^{s'} } =1}\sum_{Q \in {\Sp}_m , Q \subset U} \int_Q g\\
&\lesssim& \sup_{ {||g||}_{L^{s'} } =1} \int_{Q} Mg\\
&\le&  \sup_{ {||g||}_{L^{s'} } =1}{|Q|}^{1/s}  {||Mg||}_{L^{s'} } \\
&\le& s|Q|^{1/s},
\end{eqnarray*}
which yields
\begin{eqnarray*}
|E| \le 2^{(m-k)s/2+2s}{\gamma}^{-s/2} s^s|U|.
\end{eqnarray*}
From this and Theorem \ref{A infty}, we obtain
\begin{eqnarray*}
\sum_{k \in \Z} 2^{kp} \omega(\{ {\Lambda}_{\Sp,r ,X} f >2^k ,  \ M_r f \le 2^k\} ) &\lesssim& \sum_{ k \in \Z} 2^{kp} \sum_{2^m \le \gamma 2^k}2^{(m-k)s/(2{\delta}')+2s/ {\delta}'}  {\gamma}^{-s/(2{\delta}') }  s^{s/ {\delta}' } \sum_{U \in {\Sp}^{*}_{m} }\omega(U) \\
&\le& 2^{2s/ {\delta}'} {\gamma}^{s/(2{\delta}' )} s^{s/{\delta}'}\sum_{m \in \Z} 2^{ms/(2{\delta}') } \omega (\{ M_r f>2^m \} ) \sum_{2^k \ge{\gamma}^{-1} 2^m}2^{kp-ks/(2{\delta}')}\\
&\lesssim&2^{2s/{\delta}' } {\gamma}^{-p+s/{\delta}' }s^{s/{\delta}'} {||M_rf||}^{p}_{L^p(\omega)}
\end{eqnarray*}
for any $s/(2{\delta}')>p$. Since
\begin{eqnarray*}
{\delta}' = \frac{1+c_n {[\omega]}^{-1}_{A_{\infty}} }{c_n {[\omega]}^{-1}_{A_{\infty}}} \sim {[ \omega]}_{A_{\infty}},
\end{eqnarray*}
we obtain the desired inequality by taking $\gamma={[\omega]}^{-1}_{A_{\infty}}$ and $s=c{[\omega]}_{A_{\infty}}$ with some large constant $c>0$ depends on only $n$ and $p$.

\end{proof}

%%%%%%%%%%%%%%%%%%%%%%% Section 3 %%%%%%%%%%%%%%%%%%%%%%%%%%%%%%%%%%%%%%%%

\section{Sparse form bounds for Pseudodifferential operators}

\subsection{ Besov-type sparse form bounds}
Beltran and Cladek \cite{B-C} established sparse form bounds of pseudodifferential operators 
\begin{eqnarray*}
|\lan a(x,D)f,g \ran  | \lesssim \Lambda_{r,s'}(f,g)
\end{eqnarray*}
with $a \in S^{m}_{\rho , \rho}$ and $m< m(r,s)$ where
\[
{m(r,s) }=\begin{cases}
-n(1-\rho)(1/r-1/2) & 1\le r \le s \le 2 \\
-n(1-\rho)(1/r-1/s) & 1\le r \le 2 \le s \le r'
\end{cases}.
\]
It is natural to ask whether the such bounds hold or not when $m=m(r,s)$. However, we do not know how to settle this problem. Therefore, we treat the case $m=m(r,s)$ by using Besov type sparse form bounds
\begin{eqnarray*}
|\lan a(x,D)f , g \ran | \lesssim \sum_{j \ge 0} 2^{j \kappa} {\Lambda}_{ \Sp_j ,r,s'}( \phi_j *f,g)
\end{eqnarray*}
with suitable $\kappa \in \R$. By using Beltran and Cladek's idea, it is not hard to see 
\begin{eqnarray*}
|\lan a(x,D)f , g \ran | \lesssim \sum_{j \ge 0} 2^{j m-jm(r,s) +j \varepsilon} {\Lambda}_{ \Sp_j ,r,s'}( \phi_j *f,g)
\end{eqnarray*}
for any $\varepsilon >0$. Our purpose is to eliminate $\varepsilon$ in the above inequality. More generally, we use
\begin{eqnarray*}
{\Lambda}^{\alpha}_{\mathcal{S} , r,s}(f,g) := {\left(\sum_{Q \in \mathcal{S}} |Q| {\langle f \rangle }^{\alpha}_{r,Q}{\langle g \rangle }^{\alpha}_{s,Q}\right)}^{1/\alpha}
\end{eqnarray*}
to obtain the following results:
\begin{thm}
\label{sparseform1}
Let $2 \le s \le \infty$ and $ 2/3 < \alpha \le1$, and $ a \in S^{m}_{\rho , \rho}$ with $m \le 0$, $0< \rho <1$. Then for any $f ,g \in \sh $, there exist the sequence of sparse families ${\{\Sp_j \}}_{j =0,1,\cdots}$ such that
\begin{eqnarray*}
|\lan a(x,D)f , g \ran | \lesssim \liminf_{R \to \infty} \sum_{j \ge 0} 2^{j \kappa_1} {\Lambda}^{\alpha}_{ \Sp_j ,2,s'}( (\phi_j *f )1_{Q_R},g),
\end{eqnarray*}
where $\kappa_1= m+n(1- \rho)(1/2-1/s)+ \rho n (1/\alpha -1)$. Here, $Q_R$ denotes the cube whose center is origin and side length is $R$.
\end{thm}
\begin{thm}
\label{sparseform2}
$({\rm{i}})$ Let $ 2 \le s \le \infty$ and $s'/2 < \alpha \le 1$, and $ a \in S^{m}_{\rho , \rho}$ with $m \le 0$, $0< \rho <1$. Then for any $f ,g \in \sh $, there exist the sequence of sparse families ${\{\Sp_j \}}_{j =0,1,\cdots}$ such that
\begin{eqnarray*}
|\lan a(x,D)f , g \ran | \lesssim  \liminf_{R \to \infty} \sum_{j \ge 0} 2^{j \kappa_2} {\Lambda}^{\alpha}_{ \Sp_j ,s',s'}(( \phi_j *f) 1_{Q_R}  ,g),
\end{eqnarray*}
where $\kappa_2= m+n(1- \rho)(1-2/s)+\rho n(1/\alpha -1)$.

\noindent$({\rm{ii}})$ Let $ 1\le s' \le r \le 2 \le s \le \infty$ and $ a \in S^{m}_{\rho , \rho}$ with $m \le 0$, $0< \rho <1$. Then for any $f ,g \in \sh$, there exist the sequence of sparse families ${\{\Sp_j \}}_{j =0,1,\cdots}$ such that
\begin{eqnarray*}
|\lan a(x,D)f , g \ran | \lesssim  \liminf_{R \to \infty} \sum_{j \ge 0} 2^{j \kappa_3} {\Lambda}_{ \Sp_j ,r,s'}( (\phi_j *f)1_{Q_R},g),
\end{eqnarray*}
where $\kappa_3= m+n(1- \rho)(1/r-1/s)$.
\end{thm}
To prove Theorem \ref{sparseform1} and Theorem \ref{sparseform2}, we introduce maximal operators $M_{T,s}$ defined by
\begin{eqnarray*}
M_{T,s} f(x) := \sup_{ Q \ni x} {|Q|}^{-1/s} {|| T(f 1_{(3Q)^{c} } )||}_{ L^s (Q)}
\end{eqnarray*}
for each linear operators $T$ and $s \in[1,\infty]$.

\begin{pro}
\label{sp}
Let $1\le r <s \le \infty$ and $0< \alpha \le 1 $, and $T$ denotes the linear operators on function spaces. We assume weak-type $(r,p)$ of $T$ and $M_{T,s}$ with
\begin{eqnarray*}
\frac{1}{p} &=& \frac{1}{r}-\frac{1}{\alpha} +1 . \\
\end{eqnarray*}
Then, for any $f \in L^{\infty}_c$ and $ g \in \sh$, there exists the sparse family $\Sp$ such that
\begin{eqnarray*}
 |\langle Tf ,g \rangle | \lesssim ( {||T||}_{L^r \to L^{p,\infty} } + {|| M_{T,s}||}_{L^r \to L^{p,\infty}  })     {\Lambda}^{\alpha}_{\mathcal{S} , r,s}(f,g).
\end{eqnarray*}
\end{pro}
Proposition \ref{sp} with $\alpha =1$ was proved by Lerner in \cite{L-1}. The proposition with general $\alpha$ is proved in a similar manner, but we give the proof for reader's convenience.
\begin{lem}
\label{splem}
.Let $1\le r <s \le \infty$ and $0< \alpha \le 1$, $f \in L^{\infty}_{c}$ and $g \in \sh$, and $T$ denotes the linear operators on function spaces. We assume that for any cubes $Q \subset \Rn$ there exists some family $\mathscr{F}(Q) $ of dyadic child of $Q$ such that
\begin{eqnarray*}
&{\rm(F {\mathchar`-}1)} & \ {\mathcal{F}}_Q \ is \ pairwise \ disjoint, \\ 
&{\rm(F {\mathchar`-}2)} &  \ \sum_{ P \in {\mathcal{F}}_Q } |P| \le \frac{1}{2} |Q|,\\ 
&{\rm(F {\mathchar`-}3)} & \ \left| \int_Q T(f 1_{3Q}) g dx \right| \le ^{\exists}C {|Q|}^{1/\alpha} {\langle f \rangle}_{r,3Q} {\langle g \rangle}_{s,Q} +\sum_{ P\in {\mathcal{F}}_Q } \left| \int_P T(f 1_{3P}) g dx \right|.
\end{eqnarray*}
Then, there exists the sparse family $\Sp$ such that
\begin{eqnarray*}
 |\langle Tf ,g \rangle | \le C   {\Lambda}^{\alpha}_{\mathcal{S} , r,s}(f,g).
\end{eqnarray*}
\end{lem}

\begin{proof}
Pick up a cube $Q_0$ in $\Rn$ containing supports of $f$. Then, we construct $ {\{ {\mathcal{F} }_k \}}_{k=0,1,2, \cdots}$ by
\begin{eqnarray*}
{\mathcal{F} }_0=\{ Q_0 \} \ \ \ , \ \ \ {\mathcal{F}}_{k+1} = \bigcup_{P \in {\mathcal{F}}_k } {\mathcal{F} }_P,
\end{eqnarray*}
and set ${\mathcal{S}}_k(Q_0): =\Sp_k := \bigcup^{k}_{i=0} {\mathcal{F}}_i$, $ \Sp(Q_0):= \mathcal{S} = \bigcup_k {\mathcal{S}}_k$. From the assumption (F-1), ${\mathcal{F}}_k$ be a pairwise disjoint family. The assumption (F-3) gives 
\begin{eqnarray*}
\left| \int_{Q_0} T(f 1_{3Q_0}) g dx \right| &\le&  C \sum_{P \in {\mathcal{S}}_k } {|P|}^{1/\alpha} {\langle f \rangle}_{r,3P} {\langle g \rangle}_{s,P} +\sum_{ P\in {\mathcal{F}}_{k+1} } \left| \int_P T(f 1_{3P}) g dx \right|\\
&\le& C \sum_{P \in {\mathcal{S}} } {|P|}^{1/\alpha} {\langle f \rangle}_{r,3P} {\langle g \rangle}_{s,P} +\sum_{ P\in {\mathcal{F}}_{k+1} } \left| \int_P T(f 1_{3P}) g dx \right|
\end{eqnarray*}
for any $k \in \N$. From
\begin{eqnarray*}
\sum_{P \in {\mathcal{F}}_{k+1} } |P| \le \sum_{L \in {\mathcal{F}}_k } \sum_{ P \in {\mathcal{F}}_L } |P| \le \frac{1}{2}\sum_{L \in {\mathcal{F}}_k } |L| \le \cdots \le 2^{-k-1}|Q_0|,
\end{eqnarray*}
we have
\begin{eqnarray*}
\sum_{ P\in {\mathcal{F}}_{k+1} } \left| \int_P T(f 1_{3P}) g dx \right| \to 0 \ \ as \ k \to \infty .
\end{eqnarray*}
Therefore, one obtains
\begin{eqnarray*}
\left| \int_{Q_0} T(f 1_{3{Q_0}}) g dx \right| \le C    {\Lambda}^{\alpha}_{\mathcal{S} , r,s}(f,g).
\end{eqnarray*}
We prove the sparseness of $\mathcal{S}$. Let $Q$ be an any dyadic child of $Q_0$. For any $k$, we have
\begin{eqnarray*}
\sum_{ P \in {\mathcal{F}}_{k+1} , P \subset Q} |P| &\le&  \sum_{ L \in {\mathcal{F}}_k } \sum_{\substack{P \in {\mathcal{F}}_L \\ P \subset Q}}|P| \\ 
&\le&\sum_{ \substack{L \in {\mathcal{F}}_k \\ L \subset Q}} \sum_{\substack{P \in {\mathcal{F}}_L \\ P \subset Q}}|P| + \sum_{\substack{ L \in {\mathcal{F}}_k \\ L \supset Q }} \sum_{\substack{P \in {\mathcal{F}}_L \\ P \subset Q}}|P|\\
&\le& \frac{1}{2} \sum_{L \in {\mathcal{F}}_k,L \subset Q} |L| + \sum_{\substack{ L \in {\mathcal{F}}_k \\ L \supset Q }} \sum_{\substack{P \in {\mathcal{F}}_L \\ P \subset Q}}|P|\\
&=:& a_k+b_k.
\end{eqnarray*}
Here, if $b_k \neq 0$ for some $k$, it holds that $b_i=0$ for any $i>k$. Actually, $b_k \neq 0$ means that there are $L \in {\mathcal{F}}_k$ and $P \in {\mathcal{F}}_L \subset {\mathcal{F}}_{k+1}$ so that $L\supset Q$ and $P\subset Q$. From the pairwise  disjointness of ${\mathcal{F}}_{k+1}$, any cube in $\bigcup_{i>k} {\mathcal{F}}_{i}$ do not contain $Q$. Hence, we have $b_i=0$ with $i>k$, and 
\begin{eqnarray*}
\sum_{k\ge 0} b_k \le |Q|.
\end{eqnarray*}
From these results, one has
\begin{eqnarray*}
\sum_{k \ge 0} a_k &\le& \frac{1}{2} \sum_{k \ge 0} a_k + |Q| \\
\sum_{k \ge 0} a_k &\le& 2|Q|,
\end{eqnarray*}
which means $\Sp$ be a Carleson family, and therefore $\Sp$ be a sparse family. To complete the proof, we take the pairwise disjoint family of cubes ${\{Q_j\}}_{j=0,1,2\cdots}$ so that any $3Q_j$ contain the support of $f$ and the union of $Q_j$ coincides $\Rn$. Then, $\Sp := \cup^{\infty}_{j=0} \Sp(Q_j)$ be a sparse family, and we obtain the desired sparse form bound.

\end{proof}
Let us prove Proposition \ref{sp}.

\begin{proof}

For any cube $Q$ in $\Rn$  and $\lambda >0$, set
\begin{eqnarray*}
E=\{ x \in Q \ ; \ T(f1_{3Q}) > \lambda {|Q|}^{1/\alpha -1} {\langle f \rangle}_{r,3Q} \} \cup \{ x \in Q  \ ; \ M_{T,s} (f1_{3Q}) > \lambda {|Q|}^{1/\alpha -1} {\langle f \rangle}_{r,3Q} \}.
\end{eqnarray*}
From weak-type boundedness of  $T$ and $M_{T,s}$, we obtain
\begin{eqnarray*}
{| \{ x \in Q  \ ; \ T (f1_{3Q}) > \lambda{|Q |}^{1/\alpha -1}  {\langle f \rangle}_{r,3Q} \}|}^{1/p} &\le& {\lambda }^{-1}  {|Q|}^{1-1/\alpha } {\langle f \rangle}^{-1}_{r,3Q} {||T||}_{L^r \to L^{p,\infty} }{||f||}_{L^r(3Q)} \\
&\lesssim&  {\lambda }^{-1} {|Q|}^{1/p} ,
\end{eqnarray*}
and
\begin{eqnarray*}
{| \{ x \in Q  \ ; \ M_{T,s}  (f1_{3Q}) > \lambda {|Q|}^{1/\alpha - 1} {\langle f \rangle}_{r,3Q} \}|}^{1/q} &\le& {\lambda }^{-1} {|Q|}^{-1/\alpha +1} {\langle f \rangle}^{-1}_{r,3Q}{||M^{\gamma}_{T,s} ||}_{L^r \to L^{p,\infty} }{||f||}_{L^r(3Q)} \\
&\lesssim&  {\lambda }^{-1} {|Q|}^{1/p} .
\end{eqnarray*}
We  apply the Calderon-Zygmund decomposition to $1_E$ to construct the family ${\{ P_j \}}_j $ of pairwise disjoint dyadic child of $Q$ so that
\begin{eqnarray*}
\left\{
\begin{array}{ll}
2^{-n-1}|P_j|< |P_j \cap E| \le 2^{-1}|P_j| ,\\
|E \setminus P|=0 ,
\end{array}
\right.
\end{eqnarray*}
where $P=\bigcup P_j$. Here, the pairwise disjointness of ${\{P_j\}}_j$ gives 
\begin{eqnarray*}
\left| \int_{Q} T(f1_{3Q})g dx \right| &\le& \left| \int_{Q \setminus P} T(f1_{3Q})g dx \right| +  \sum_j \left| \int_{P_j} T(f1_{3Q\setminus 3P_j})g dx \right| + \sum_j \left| \int_{P_j} T(f1_{3P_j})g dx \right| \\
&=:& I_1 +I_2 +I_3.
\end{eqnarray*}
Since $|E \setminus P|=0$, one obtains
\begin{eqnarray*}
I_1 \le \int_{Q \setminus E} |T(f1_{3Q})||g| dx \lesssim \lambda {|Q|}^{1/\alpha -1}{\lan f \ran}_{r,3Q}\int_{Q} |g| \le  \lambda |Q|^{1/\alpha} {\langle f \rangle}_{r,3Q} {\langle g \rangle}_{s',Q} .
\end{eqnarray*}
On the other hands,
\begin{eqnarray*}
I_2 &\le& \sum_j {||T(f1_{3Q  \setminus 3P_j}) ||}_{L^s(P_j)} {||g||}_{L^{s'}(P_j)} \\
&\le& { \left( \sum_j {||T(f1_{3Q  \setminus 3P_j}) ||}^s_{L^s(P_j)} \right)}^{1/s} {||g||}_{L^{s'}(Q)} \\
&\lesssim&  { \left( \sum_j {|P_j|} \right)}^{1/s} {|Q|}^{1/\alpha - 1} {\lan f \ran}_{r,3Q} {||g||}_{L^{s'}(Q)} \\
&\le& \lambda |Q|^{1/\alpha} {\langle f \rangle}_{r,3Q} {\langle g \rangle}_{s',Q}.
\end{eqnarray*}
From these results with sufficient large $ \lambda \sim  ( {||T||}_{L^r \to L^{p,\infty} } + {|| M_{T,s}||}_{L^r \to L^{p,\infty}  }) $ and Lemma \ref{splem}, we obtain $\sum |P_j| < 2^{-1} |Q|$ and complete the proof.
\end{proof}

\begin{rem}
From Lebesgue's differentiation theorem, we obtain
\begin{eqnarray*}
|Tf(x)| &=& \lim_{  \substack{|Q| \to 0 \\ Q \ni x}   }{\left( \frac{1}{|Q|} \int_{Q} { |Tf|}^{s} \right) }^{1/s}  \\
&\le& M_{T,s}f(x) + \liminf_{  \substack{|Q| \to 0 \\ Q \ni x}   }{\left( \frac{1}{|Q|} \int_{Q} { |T(f1_{(3Q)})|}^{s} \right) }^{1/s}.
\end{eqnarray*}
If $T$ is a bounded operator from $L^{s -\varepsilon}$ to $L^s$ with some $\varepsilon >0$, then we have
\begin{eqnarray*}
\liminf_{  \substack{|Q| \to 0 \\ Q \ni x}   }{\left( \frac{1}{|Q|} \int_{Q} { |T(f1_{(3Q)})|}^{s} \right) }^{1/s} &\lesssim& \liminf_{  \substack{|Q| \to 0 \\ Q \ni x}   }  \frac{1}{{|Q|}^{1/s}} { \left( \int_{3Q} { |f |}^{s-\varepsilon} \right) }^{1/(s-\varepsilon)}\\
&\lesssim&\liminf_{  \substack{|Q| \to 0 \\ Q \ni x}   }  {|Q|}^{1/(s-\varepsilon) -1/s} { \left(\frac{1}{|3Q|} \int_{3Q} { |f |}^{s-\varepsilon} \right) }^{1/(s-\varepsilon)}\\
&=&0.
\end{eqnarray*}
Hence, we have $|Tf(x)| \le M_{T,s} f(x)$ and ${||T ||}_{L^r \to L^{p,\infty}} \le {||M_{T,s} ||}_{ L^r \to L^{p,\infty}}$.

\end{rem}
%%%%%%%%%%%%%%%%%%%%%%%%%%%% Interpolation %%%%%%%%%%%%%%%%%%%%%%%%%%%%%%%%%%%%%%%%%%%55
The Proposition \ref{sp} gives some interpolation theorem.
\begin{cor}
\label{inter1}
Let $1 \le r \le s_0 ,  s_1,p_0,p_1 \le \infty$. We assume linear operator $T$ satisfies
\begin{eqnarray*}
{||M_{T,s_0} f||}_{L^{p_0, \infty }} &\le& C_0 {||f||}_{L^r}, \\
{||M_{T,s_1} f||}_{L^{p_1,\infty} }  &\le& C_1 {||f||}_{L^r}.
\end{eqnarray*}
Then, for any $\theta \in (0,1)$, we have
\begin{eqnarray*}
{||M_{T,s}||}_{L^r \to L^{p,\infty} }\lesssim   {C_0}^{1- \theta} {C_1}^{\theta} 
\end{eqnarray*}
where $1/s =(1-\theta)/s_0 + \theta /s_1$ and $1/p=(1-\theta)/p_0+\theta/p_1$. In particularly, we have
\begin{eqnarray*}
|\lan Tf ,g \ran | \lesssim  ({||T||}_{L^r \to L^{p,\infty}} + {C_0}^{1- \theta} {C_1}^{\theta} ) \Lambda^{\alpha}_{ \Sp ,r ,s'}(f,g),
\end{eqnarray*}
where 
\begin{eqnarray*}
\frac{1}{\alpha} = \frac{1}{r} +\frac{1}{p'}.
\end{eqnarray*}
\end{cor}

\begin{proof}
Let $Q$ and $x\in Q$. For any simple functions $f,g $ so that ${||g||}_{s'}=1$, we define the analytic function $F$ on the open strip by
\begin{eqnarray*}
F(z)=\int_Q T(f1_{{(3Q)}^c } )(x) g_z(x) dx ,
\end{eqnarray*}
where
\begin{eqnarray*}
g_z=sgn(g) {|g|}^{s'\{ (1-z)/s_0'+z/s_1'\} }.
\end{eqnarray*}
Then, it holds that
\begin{eqnarray*}
|F(iy)| &\le& \int_Q |T(f1_{ {(3Q)}^c})| {|g|}^{s'/s_0'} \\
&\le& {||T(f 1_{ {(3Q)}^c } )||}_{L^{s_0} (Q)}\\
&\le& {|Q|}^{1/s_0} M_{T,s_0}f(x).
\end{eqnarray*}
On the other hands, one has
\begin{eqnarray*}
|F(1+iy)| &\le& \int_Q |T(f1_{ {(3Q)}^c})| {|g|}^{s'/s_1'} \\
&\le& {||T(f 1_{ {(3Q)}^c } )||}_{L^{s_1}(Q)}\\
&\le& {|Q|}^{1/ s_1} M_{T,s_1} f(x). 
\end{eqnarray*}
By using Hadamard's three lines lemma, we have
\begin{eqnarray*}
|F(\theta)| &\le& {|Q|}^{ (1-\theta)/s_0+\theta/s_1} {M_{T,s_0} f (x) }^{1-\theta} {M_{T,s_1}f(x)}^{\theta}\\
&=& {|Q|}^{ 1/s } {M_{T,s_0} f (x) }^{1-\theta}  {M_{T,s_1}f(x)}^{\theta} , \\
\end{eqnarray*}
which yields
\begin{eqnarray*}
M_{T,s} f(x) \le  {M_{T,s_0}f (x) }^{1-\theta}  {M_{T,s_1}f(x)}^{\theta}.
\end{eqnarray*}
By H\"{o}lder's inequality, we have
\begin{eqnarray*}
{||{(M_{T,s_0}  f)}^{1-\theta}  {(M_{T,s_1}f)}^{\theta} ||}_{L^{p, \infty} } &\lesssim& {||M_{T,s_0} f||}^{1-\theta}_{L^{p_0,\infty} } {||M^{\gamma}_{T,s_1} f||}^{\theta}_{L^{p_1,\infty} } \\
&\le& {C_0}^{1-\theta} {C_1}^{\theta} {||f||}_{L^r}
\end{eqnarray*}
for $1/p=(1-\theta)/p_0+ \theta /p_1$. By this and Proposition \ref{sp}, we have $ {||M^{\gamma}_{T,s}||}_{L^r \to L^{p,\infty} }\lesssim   {C_0}^{1- \theta} {C_1}^{\theta} $ and the desired sparse form bounds for $T$.
\end{proof}

\begin{cor}
\label{inter2}
Let $1 \le r_0 ,r_1\le s_0 ,  s_1,p_0,p_1 \le \infty$. We assume linear operator $T$ satisfies
\begin{eqnarray*}
{||M_{T,s_0} f||}_{L^{p_0, \infty }} &\le& C_0 {||f||}_{L^{r_0}} ,\\
{||M_{T,s_1} f||}_{L^{p_1,\infty} }  &\le& C_1 {||f||}_{L^{r_1}}.
\end{eqnarray*}
and
\begin{eqnarray*}
|Tf(x)| \le T(|f|)(x)  \ \ a.e. \ x \in \Rn .
\end{eqnarray*}
Then, for any $\theta \in (0,1)$, we have
\begin{eqnarray*}
{||M_{T,s}||}_{L^r \to L^{p,\infty} }\lesssim   {C_0}^{1- \theta} {C_1}^{\theta} ,
\end{eqnarray*}
where $1/s =(1-\theta)/s_0 + \theta /s_1$ and $1/q=(1-\theta)/p_0+\theta/p_1$. In particularly, we have
\begin{eqnarray*}
|\lan Tf ,g \ran | \lesssim  ({||T||}_{L^r \to L^{p,\infty}} + {C_0}^{1- \theta} {C_1}^{\theta} ) \Lambda^{\alpha}_{ \Sp ,r ,s'}(f,g),
\end{eqnarray*}
where
\begin{eqnarray*}
\frac{1}{\alpha} = \frac{1}{r} +\frac{1}{p'}.
\end{eqnarray*}
\end{cor}

\begin{proof}
The proof is similar to that of Corollary \ref{inter1}. Let $Q$ and $x\in Q$. For any simple functions $f,g $ so that ${||g||}_{s'}=1$, we define the analytic function $F$ on the open strip by
\begin{eqnarray*}
F(z)=\int_Q T(f_z 1_{{(3Q)}^c } )(x) g_z(x) dx,
\end{eqnarray*}
where
\begin{eqnarray*}
f_z&=& sgn(f) {|f|  }^{ r \{ (1-z)/r_0+z/r_1 \} } , \\
g_z&=&sgn(g) {|g|}^{s'\{ (1-z)/s_0'+z/s_1'\} }.
\end{eqnarray*}
From $|Tf | \le T(|f|)$, we have
\begin{eqnarray*}
|F(iy)| &\le& \int_Q |T(f1_{ {(3Q)}^c})| {|g|}^{s'/s_0'} \\
&\le& {||T(f_z 1_{ {(3Q)}^c } )||}_{L^{s_0} (Q)}\\
&\le& {|Q|}^{1/s_0} M_{T,s_0}({|f|}^{r/r_0} )(x),
\end{eqnarray*}
and
\begin{eqnarray*}
|F(1+iy)| &\le& \int_Q |T(f1_{ {(3Q)}^c})| {|g|}^{s'/s_1'} \\
&\le& {||T(f_z 1_{ {(3Q)}^c } )||}_{L^{s_1}(Q)}\\
&\le& {|Q|}^{1/s_1} M_{T,s_1} ({|f|}^{r/r_1} )(x). 
\end{eqnarray*}
Hence, one obtains
\begin{eqnarray*}
M_{T,s} f(x) \le {M_{T,s_0}( {|f|}^{r/r_0} )(x) }^{1-\theta}  {M_{T,s_1}({|f|}^{r/r_1} )(x)}^{\theta}. \\
\end{eqnarray*}
By using H\"{o}lder's inequality, we have
\begin{eqnarray*}
{||{(M_{T,s_0} ({|f|}^{r/r_0} ))}^{1-\theta}  {(M_{T,s_1}({|f|}^{r/r_1} ))}^{\theta} ||}_{L^{p, \infty} } &\lesssim& {||M_{T,s_0}({|f|}^{r/r_0}) ||}^{1-\theta}_{L^{p_0,\infty} } {||M_{T,s_1}( {|f|}^{r/r_1}) ||}^{\theta}_{L^{p_1,\infty} } \\
&\le& {C_0}^{1-\theta} {C_1}^{\theta} {||f||}_{L^r}.
\end{eqnarray*}

\end{proof}
%%%%%%%%%%%%%%%%%%%%%%%%%%%%%%%% Sparse form1 %%%%%%%%%%%%%%%%%%%%%%%%%%%%%%%%%%%%%%%%%%%%%%%
We give a proof of Theorem \ref{sparseform1}.
\begin{proof}
We recall the dyadic decomposition in subsection 2.1. Since $ \phi_j *f= (\phi_{j-1} +\phi_j +\phi_{j+1}) * \phi_j *f$, we have
\begin{eqnarray*}
 |\lan a(x,D)f,g \ran| &=&\sum_{j \ge0} \sum^{j+1}_{i=j-1} |\lan a_i(x,D)(\phi_j*f),g \ran| \\
&=& \sum_{j \ge0} \sum^{j+1}_{i=j-1} |\lan a_i(x,D)(\lim_{R\to \infty } 1_{Q_R}\phi_j*f),g \ran| \\
&\le& \liminf_{R\to \infty} \sum_{j \ge0} \sum^{j+1}_{i=j-1} |\lan a_i(x,D)((\phi_j*f) 1_{Q_R} ),g \ran| .
\end{eqnarray*}
Therefore, it is enough to prove
\begin{eqnarray*}
| \lan a_j(x,D) f ,g \ran | \lesssim 2^{j \kappa_1} \Lambda^{\alpha}_{\Sp_j, 2,s'}(f,g)
\end{eqnarray*}
for any $f \in L^{\infty}_{c}$ and $ g \in \sh$. For any $x, z \in Q$ and $\gamma \in [0, 1)$, we integrate by parts $N \in \N$ times to obtain
\begin{eqnarray*}
 |a_j(x,D)(f1_{{(3Q)}^c } )(z)| &\lesssim& 2^{jm+jn/2} { \left\{  \int_{{(3Q)}^c}  {(1+2^{2j\rho N }{ |z-y|}^{2N})}^{-2} {|f(y)|}^2 dy \right\} }^{1/2} \\
&\lesssim& 2^{jm+jn/2} { \left\{  \int {(1+2^{2j\rho N }{ |x-y|}^{2N})}^{-2} {|f(y)|}^2 dy \right\} }^{1/2} \\
&\lesssim& 2^{jm+jn/2} \sum_{k \in \mathbb{Z} }  { \left\{  \int_{ |x-y| \sim 2^{-j \rho } 2^k} {(1+ 2^{2kN})}^{-2} {|f(y)|}^2 dy \right\} }^{1/2} \\
&\lesssim&2^{jm+jn(1-\rho(1-\gamma))/2} {M^{\gamma}({|f|}^2)(x)}^{1/2},
\end{eqnarray*}
where $M^{\gamma}h(x):= \sup_{Q\in x} {|Q|}^{\gamma } {\lan h\ran}_Q$. Hence, we obtain
\begin{eqnarray*}
 M_{a_j(x,D) ,\infty }\lesssim 2^{jm+jn(1-\rho(1-\gamma))/2} {M^{\gamma}({|f|}^2)(x)}^{1/2}.
\end{eqnarray*}
By weak-type boundedness of $M^{\gamma}$, for any $p_0 \ge 2$, one has
\begin{eqnarray*}
 {|| M_{a_j(x,D) ,\infty } ||}_{L^{p_0,\infty}}  \lesssim 2^{jm+jn(1-\rho)/2 +j\rho n(1/2 -1/p_0)}  {||f||}_{L^2}
\end{eqnarray*}
by taking $\gamma =1- 2/p_0$. On the other hands, we have
\begin{eqnarray*}
{||a_j(x,D)f||}_{L^{p_1}} \lesssim 2^{jm +jn(1/2 -1/p_1)} {||f||}_{L^2}
\end{eqnarray*}
for any $p_1 \ge2$. From this and
\begin{eqnarray*}
M_{a_j(x,D) , p_1}f (x) \lesssim M_{p_1}(a_j(x,D)f)(x) + 2^{jm +jn(1/2-1/p_1)}{M^{1-2/p_1} ({|f|}^2) (x)}^{1/2},
\end{eqnarray*}
we obtain
\begin{eqnarray*}
{||M_{a_j(x,D) , p_1}f ||}_{L^{p_1,\infty}} \lesssim  2^{jm +jn(1/2-1/p_1)} {||f||}_{L^2}.
\end{eqnarray*}
Therefore, Corollary \ref{inter1} gives 
\begin{eqnarray*}
 |\lan a_j(x,D)f,g \ran| \lesssim 2^{jm} 2^{jn(1-\theta)(1-\rho)/2+j\rho n(1-\theta)(1/2-1/p_0) } 2^{jn\theta(1/2-1/p_1)  }\Lambda^{\alpha}_{ \Sp ,2 ,s'}(f,g),
\end{eqnarray*}
with $1/s=\theta/p_1$ and $1/\alpha= 1/2-(1-\theta)/p_0-\theta/p_1 +1<3/2$. By simple calculation as following,
\begin{eqnarray*}
(1-\theta)(1-\rho)/2 +\rho(1-\theta)(1/2 -1/p_0) + \theta(1/2-1/p_1)&=&-\rho(1-\theta)/p_0 +1/2-\theta/p_1 \\
&=&-\rho(1-1/\alpha +1/2-1/s) +1/2-1/s \\
&=&(1-\rho)(1/2-1/s)+\rho (1/\alpha-1),
\end{eqnarray*}
we have the desired sparse bounds.
\end{proof}
%%%%%%%%%%%%%%%%%%%%%%%%%%%%%%%%%%%%%%% Interpolation %%%%%%%%%%%%%%%%%%%%%%%%%%%%%%%%%%%%%%%%%
To establish Theorem \ref{sparseform2} by the interpolation argument as Corollary \ref{inter2}, we need the condition $|a(x,D)f| \le a(x,D)(|f|)$. Unfortunately, it fails in general and we need the following alternative argument:

\begin{lem}
\label{inter3}
Let $0 \le \gamma <1$. We assume linear operator $T$ satisfies
\begin{eqnarray*}
{||T||}_{L^2 \to L^2} &\le& C_0 , \\
M_{T,\infty} f(x)  &\le& C_1 {M^{\gamma} f(x)} \ \ a.e. \ x \in \Rn.
\end{eqnarray*}
Then, for any $\theta \in (0,1)$, we have
\begin{eqnarray*}
{||M_{T,r'}||}_{L^r \to L^{p,\infty} }\lesssim   {C_0}^{1- \theta} {C_1}^{\theta} ,
\end{eqnarray*}
where $1/r =(1-\theta)/2 + \theta $ and $1/p=(1-\gamma)\theta +1/r'$. In particularly, we have
\begin{eqnarray*}
|\lan Tf ,g \ran | \lesssim  ({||T||}_{L^r \to L^{p,\infty}} + {C_0}^{1- \theta} {C_1}^{\theta} ) \Lambda^{\alpha}_{ \Sp ,r ,r}(f,g),
\end{eqnarray*}
where
\begin{eqnarray*}
\frac{1}{\alpha} = 1+\left(\frac{2}{r}-1\right) \gamma.
\end{eqnarray*}
\end{lem}

\begin{proof}
We put $E=\{ M_{T,r'}f > \lambda \}$ for any $\lambda >0$. For each $\delta >0$, we have
\begin{eqnarray*}
|E| &\le& |\{ M_{T,r'} f> \lambda , \  M^{\gamma}_{r} f \le \delta \lambda \} |+| \{ M^{\gamma}_{r}f >\delta \lambda  \}| \\
&=:&|E_0|+|E_1|,
\end{eqnarray*}
where $M^{\gamma}_{r} f= {M^{\gamma} ({|f|}^r) }^{1/r}$. Weak-type boundedness of $M^{\gamma}_{r} $ gives  
\begin{eqnarray*}
{|E_1|}^{1/q} \lesssim \delta^{-1} {\lambda}^{-1} {||f||}_{L^r},
\end{eqnarray*} 
with $1/q=1/r - \gamma/r$. We need to estimate the $|E_0|$. For any $x \in E_0$, there exists a cube $Q_x$ such that
\begin{eqnarray*}
|Q_x| < \lambda^{r'} {||T(f1_{{(3Q_x)}^c} )||}^{r'}_{L^{r'} (Q_x)}.
\end{eqnarray*} 
Let $K \subset E_0$ be an any compact set, then we can select finite pairwise disjoint subcollection ${ \{3Q_j \}}_j \subset {\{ 3Q_x\}}_{x\in E}$ such that
\begin{eqnarray*}
{|K|} \lesssim \sum_{j} |Q_j|.
\end{eqnarray*}
From the duality of $\ell^{r'}(\N ;L^{r'})$, we obtain 
\begin{eqnarray*}
{|K|}^{1/r'} &\le& {\lambda}^{-1} {\left(\sum_j {||T(f1_{ {(3Q_j)}^c})||}^{r'}_{L^{r'}(Q_j)}\right)}^{1/r'}\\
&=& \lambda^{-1}  \sup_{{\{g_j\}}_j} \left| \sum_j\int_{Q_j} T(f1_{{(3Q_j)}^c})g_j\right|.
\end{eqnarray*}
Here, the supremum is taken all over the $g= { \{g_j\}}_j$ such that ${||g||}_{\ell^{r}(\N ;L^r)} \le 1$. We define the analytic function $F$ on the open strip by
\begin{eqnarray*}
F(z)=\sum_j \int_{Q_j} T(f_z 1_{{(3Q_j)}^c } )(x) g_{z,j}(x) dx,
\end{eqnarray*}
where
\begin{eqnarray*}
f_z&=& sgn(f) {|f|  }^{ r \{ (1-z)/2+z \} },  \\
g_{z,j}&=&sgn(g_j) {|g_j|}^{r \{ (1-z)/2+z \} }.
\end{eqnarray*}
By $L^2$ boundness of $T$, one has
\begin{eqnarray*}
|F(iy)| &\le& \sum_j {||Tf_z||}_{L^2(Q_j)} {||g_j||}^{r/2}_{L^{r}} + C_0 \sum_j{||f_z||}_{L^2(3Q_j)} {||g_j||}^{r/2}_{L^{r}}\\
&\le& {||Tf_z||}_{L^2}+C_0 {||f_z||}_{L^2} \\
&\lesssim& C_0 {||f||}^{r/2}_{L^r} .
\end{eqnarray*}
Since $Q_j \cap E_0 \neq \emptyset $, we obtain
\begin{eqnarray*}
|F(1+iy)| &\le& \sum_j  \inf_{x \in Q_j} M_{T,\infty}f_z(x) {||g_j||}^{r}_{L^r}\\
 &\le& C_1 \sum_j  \inf_{x \in Q_j}{ M^{\gamma}_r f (x) }^r {||g_j||}^{r}_{L^r}\\
 &\le&C_1 {\delta}^r {\lambda}^r.
\end{eqnarray*}
By these results, we have
\begin{eqnarray*}
{|K|} &\le& {(C^{1-\theta}_0 C^{\theta}_{1} )}^{r'}\delta^{rr'\theta } \lambda^{rr' \theta-r'}{||f||}^{rr'(1-\theta)/2}_{L^r} \\
&=& {(C^{1-\theta}_0 C^{\theta}_{1} )}^{r'}\delta^{rr'\theta } \lambda^{-r }{||f||}^{r}_{L^r} , \\
\end{eqnarray*}
and
\begin{eqnarray*}
|E| \le \delta^{-q} {\lambda}^{-q} {||f||}^{q}_{L^r} + {(C^{1-\theta}_0 C^{\theta}_{1} )}^{r'}\delta^{rr'\theta } \lambda^{-r }{||f||}^{r}_{L^r}. \\
\end{eqnarray*}
Here, we optimize for $\delta$ to obtain
\begin{eqnarray*}
|E|^{1/p} \le \lambda^{-1}  {||f||}_{L^r} , \\
\end{eqnarray*}
where $1/p =r\theta /q +1/r'=(1-\gamma)\theta +1/r'$. Hence, $M_{T,r'}$ be a weak-type $(r,p)$ operator which yields 
\begin{eqnarray*}
|\lan Tf ,g \ran| \lesssim ({||T||}_{L^r \to L^{p,\infty} }+C^{1-\theta}_{0} C^{\theta}_{1} ) \Lambda^{\alpha}_{\Sp ,r,r}(f,g).
\end{eqnarray*}
\end{proof}
%%%%%%%%%%%%%%%%%%%%%%%%%%%%%%%%%% Sparse form2 %%%%%%%%%%%%%%%%%%%%%%%%%%%%%%%%%%%%%%%%%%%%55
Let us prove the Theorem \ref{sparseform2}.
\begin{proof}
The theorem follows from the pointwise estimate
\begin{eqnarray*}
M_{a_j(x,D) ,\infty} f(x) \lesssim 2^{jm+jn(1-\rho(1-\gamma)) } M^{\gamma}f(x),
\end{eqnarray*}
Lemma \ref{inter3} and Marchinkiewicz interpolation theorem. Indeed, this estimate and Lemma \ref{inter3} yield
\begin{eqnarray*}
{||M_{a_j(x,D) ,s}||}_{L^{s'} \to L^{s',\infty} } \lesssim 2^{jm+jn(1-\rho)(1-2/s)+j \rho  n (1/\alpha -1)} 
\end{eqnarray*}
by taking $1/\alpha= 1+(2/r-1)\gamma$. Moreover, by interpolating this with $\alpha = 1$ and ${||M_{a_j(x,D),s}||}_{L^2 \to L^{2,\infty} } \lesssim 2^{jm+jn(1-\rho)(1/2-1/s)}$, we have
\begin{eqnarray*}
{||M_{a_j(x,D) ,s}||}_{L^{r} \to L^{r,\infty} } \lesssim 2^{jm+jn(1-\rho)(1/r-1/s)}.
\end{eqnarray*}
Thus, we obtain the desired sparse form bounds. Now, we prove the above pointwise estimate. For any $x, z \in Q$ and $\gamma \in [0, 1)$, we integrate by parts $N \in \N$ times to obtain
\begin{eqnarray*}
 |a_j(x,D)(f1_{{(3Q)}^c } )(z)| &\lesssim& 2^{jm+jn}   \int_{{(3Q)}^c}  {(1+2^{2j\rho N }{ |z-y|}^{2N})}^{-1}  {|f(y)|}  dy  \\
&\lesssim& 2^{jm+jn} \int {(1+2^{2j\rho N }{ |x-y|}^{2N})}^{-1} {|f(y)|} dy \\
&\lesssim&2^{jm+jn(1-\rho(1-\gamma)) }  M^{\gamma}f(x).
\end{eqnarray*}
Hence, we obtain
\begin{eqnarray*}
 M_{a_j(x,D) ,\infty } (x) \lesssim 2^{jm+jn(1-\rho(1-\gamma))} M^{\gamma}f(x) ,
\end{eqnarray*}
and complete the proof.
\end{proof}
%%%%%%%%%%%%%%%%%%%%%%%%%%%%%%% Weight bounds for sparse form %%%%%%%%%%%%%%%%%%%%%%%%%%%%%%%%%%%%%%%%%%%%%%%%%%%%%%55
\subsection{Application to the boundedness on weighted Besov spaces}
This section is devoted to obtain the boundedness on weighted Besov space of pseudodifferential operators. To do this, we establish the weighted bounds for ${\Lambda}^{\alpha}_{\Sp ,r,s'}$ by using Bernicot, Frey and Petermichl's idea in \cite{B-F-P}. 
\begin{pro}
\label{weight}
Let $1\le r <q \le p  <s\le \infty$ and $1/\alpha=1/p'+1/q$. We assume the weight $\omega$ satisfy $ {\omega}^q \in A_{q/r} \cap RH_{(p/q)(s/p)'}$. Then, for any sparse family $\Sp \subset \D$ with some dyadic lattice $\D$, we have
\begin{eqnarray*}
{\Lambda}^{\alpha}_{\Sp,r,s'} (f,g) \lesssim {({[{\omega}^q]}_{A_{q/r}} {[ {\omega}^q] }_{{RH}_{(p/q)( s/p)' }})}^{\delta} {||f||}_{L^q({\omega}^q)} {||g||}_{L^{p'}({\omega}^{-p'})},
\end{eqnarray*}
where
\begin{eqnarray*}
\delta=\max \left\{ \frac{1}{q-r}, \frac{p(s-1)}{q(s-p)} \right\}.
\end{eqnarray*}

\end{pro}

\begin{proof}
We set
\begin{eqnarray*}
\mu=\omega^{-rq/(q-r)}  \ \ and  \ \ \nu= \omega^{p's'/(p'-s')}.
\end{eqnarray*}
Furthermore, let us define
\begin{eqnarray*}
F_Q= { \left( \frac{1}{\mu(Q) } \int_Q {|f|}^r  \right) }^{1/r} \ \ and  \ \ G_Q={ \left( \frac{1}{\nu(Q)} \int_Q {|g|}^{s'} \right)}^{1/s'}.
\end{eqnarray*}
Then, we have 
\begin{eqnarray*}
{\Lambda}^{\alpha}_{\Sp,r,s'} (f,g) \le {\left( \sum_{Q \in {\Sp}}  |Q|  {\lan \mu \ran}^{\alpha /r}_Q {\lan \nu \ran}^{\alpha / s'}_Q {F_Q}^{\alpha} {G_Q}^{\alpha} \right)}^{1/\alpha}.
\end{eqnarray*}
We estimate $|Q|   {\lan \mu \ran}^{\alpha /r}_Q {\lan \nu \ran}^{\alpha / s'}_Q $. By taking 
\begin{eqnarray*}
\beta = 1+ \frac{1/r-1/q}{1/p-1/s},
\end{eqnarray*}
we obtain $\mu = {\nu}^{1-\beta'}$ and $ {\lan \nu \ran}_Q {\lan \mu \ran}^{\beta -1}_{Q} \le {[\nu]}_{A_{\beta}}$. Here, we assume
\begin{eqnarray*}
\frac{1}{q-r} \le \frac{p(s-1)}{q(s-p)},
\end{eqnarray*}
which gives $\gamma := 1/r -(\beta -1)/s' \le 0$. From this assumption and the sparseness of $\Sp$, one obtains
\begin{eqnarray*}
 |Q| {\lan \mu \ran}^{\alpha /r}_Q {\lan \nu \ran}^{\alpha / s'}_Q &\le& {[\nu]}^{\alpha /s'}_{A_\beta} |Q| {\lan \mu \ran}^{\alpha \gamma}_{Q} \\
 &\le& {[\nu]}^{\alpha /s'}_{A_\beta} {|E_Q|}^{1-\alpha \gamma} {\left( \int_{E_Q} \mu \right)}^{ \alpha \gamma}.
\end{eqnarray*}
On the other hands, it holds that ${\mu}^{-\gamma }{\mu}^{1/q} {\nu}^{1/p'}=1$ since $\nu= {\mu}^{1- \beta}$. Hence, by setting $1/t=1/q+1/p' -\gamma=1 /\alpha -\gamma$ and using H\"{o}lder's inequality, we have
\begin{eqnarray*}
|E_Q|^{1/t} &=& {|| {\mu}^{-\gamma }{\mu}^{1/q} {\nu}^{1/p'}||}_{L^t(E_Q)} \\
&\le& {\mu(E_Q)}^{-\gamma} {\mu(E_Q)}^{1/q} {\nu(Q)}^{1/p'},
\end{eqnarray*}
which yields
\begin{eqnarray*}
 |Q| {\lan \mu \ran}^{\alpha /r}_Q {\lan \nu \ran}^{\alpha / s'}_Q &\le& {[\nu]}^{\alpha /s'}_{A_\beta}  {\mu(E_Q)}^{\alpha/q} {\nu(Q)}^{\alpha/p'}.
\end{eqnarray*}
From these results, we obtain
\begin{eqnarray*}
{\Lambda}^{\alpha}_{\Sp,r,s'} (f,g) &\le& {[\nu]}^{1/s'}_{A_\beta}  {\left( \sum_{Q \in {\Sp}}  {({F_Q} {\mu(E_Q)}^{1/q} G_Q {\nu(E_Q)}^{1/p'}  )}^{\alpha}\right)}^{1/\alpha} \\
&\le& {[\nu]}^{1 /s'}_{A_\beta}  {\left( \sum_{Q \in {\Sp}}  {F_Q}^q {\mu(E_Q)} \right)}^{1/q}  {\left( \sum_{Q \in {\Sp}}   {G_Q}^{p'}  {\nu(E_Q)} \right)}^{1/p'} \\
&\le& {[\nu]}^{1 /s'}_{A_\beta} { \left( \int {|M^{\D}_{r,\mu } (f {\mu}^{-1/r})|}^q d \mu \right)}^{1/q} { \left( \int {|M^{\D}_{s',\nu } (g {\nu}^{-1/s'})|}^{p'} d \nu \right)}^{1/p'} \\
&\lesssim&{[\nu]}^{1 /s'}_{A_\beta} {||f||}_{L^q({\omega}^q)} {||g||}_{L^{p'}({\omega}^{-p'} )}.
\end{eqnarray*}
In another case, by using
\begin{eqnarray*}
 |Q| {\lan \mu \ran}^{\alpha /r}_Q {\lan \nu \ran}^{\alpha / s'}_Q &\le& {[\nu]}^{\alpha /\{r(\beta -1) \} }_{A_\beta} |Q| {\lan \nu \ran}^{\alpha /s' - \alpha/\{r(\beta-1) \} }_{Q} \\
 &\le&  {[\nu]}^{\alpha /\{r(\beta -1) \} }_{A_\beta} {|E_Q|}^{1-\alpha \gamma} {\left( \int_{E_Q} \mu \right)}^{ \alpha \gamma},
\end{eqnarray*} 
and the same discussion as above, we have
\begin{eqnarray*}
{\Lambda}^{\alpha}_{\Sp,r,s'} (f,g) \lesssim   {[\nu]}^{1 /\{r(\beta -1) \} }_{A_\beta} {||f||}_{L^q({\omega}^q)} {||g||}_{L^{p'}({\omega}^{-p'})}.\end{eqnarray*}
Concluding these results, we have
\begin{eqnarray*}
{\Lambda}^{\alpha}_{\Sp,r,s'} (f,g) \lesssim   {[\nu]}^{\delta }_{A_\beta} {||f||}_{L^q({\omega}^q)} {||g||}_{L^{p'}({\omega}^{-p'})},
\end{eqnarray*}
where 
\begin{eqnarray*}
\delta =\max \left\{  \frac{q(s-p)}{ps(q-r)},\frac{s-1}{s}  \right\}.
\end{eqnarray*}
To complete the proof, we need to estimate $[\nu]_{A_{\beta}}$. However, it is deduced from the simple calculation. The detail is the following:
\begin{eqnarray*}
{\lan \nu \ran}_Q {\lan \mu \ran}^{\beta-1}_{Q} &=& {\lan {\omega}^{q \cdot {p(s/p)'/q}} \ran}_Q {\lan {\omega}^{q \cdot(-r/(q-r))} \ran}^{ p(s/p)' /q \cdot (q/r -1)}_{Q} \\
&\le& {( {[ {\omega}^q]}_{ {RH}_{(p/q)(s/q)'} } {\lan {\omega}^q \ran }_Q  {\lan {\omega}^{q \cdot(-r/(q-r))} \ran}^{  q/r -1}_{Q} )}^{p(s/p)'/q}\\
&\le&{ ( { [\omega^q ]}_{ {RH}_{(p/q)(s/p)' }} {[\omega^q ]}_{A_{q/r} })}^{ps/\{q(s-p)\} }.
\end{eqnarray*}
\end{proof}
%%%%%%%%%%%%%%%%%%%%%%%%%%%%%%%%% Weighted Besoov %%%%%%%%%%%%%%%%%%%%%%%%%%%%%%%%%%%%%%%%%
We define the weighted Besov spaces according to Bui \cite{Bu}. Suppose $0< p, \sigma < \infty$ and $\kappa \in \R$, then weighted Besov spaces $B^{\kappa}_{p,q}(\omega)$ are defined by
\begin{eqnarray*}
B^{\kappa}_{p,q}(\omega) &=& \{ f \in {\mathscr{S}}' \ ; \ {||f||}_{B^{\kappa}_{p,q}(\omega) } < \infty \} ,\\
 {||f||}_{B^{\kappa}_{p,\sigma }(\omega)}&=& {\left( \sum_{j \ge 0} 2^{j\kappa \sigma } {|| \phi_j * f||}^{\sigma}_{L^p(\omega) }  \right)}^{1/\sigma}
\end{eqnarray*}
for any $ \omega \in A_{\infty}$. Bui showed that $ \mathscr{S} $ is dense subset of $B^{\kappa}_{p,\sigma}(\omega)$. Hence, the Theorem \ref{sparseform1} and Theorem \ref{sparseform2}, and Proposition \ref{weight} give the following results about boundedness of pseudodifferential operators on weighted Besov spaces.
\begin{cor} 
\label{pse Besov}
Let $a\in S^{m}_{\rho,\delta} $ with $m\in \R$ and $0\le \delta <\rho \le 1$. Then, we have the following bounds.

\noindent$({\rm{i}})$ Let $2< q  \le p <\infty$ and $\omega^q \in A_{q/2} \cap {RH}_{(p/q)(s/p)'}$ with some $s\in(p ,\infty]$. Then, for any $\kappa \in \R$ and $0<\sigma < \infty$, $a(x,D)$ be a bounded operators from $B^{\kappa +   {\tilde{\kappa}}_1  }_{q,\sigma}(\omega^q)$ to $ B^{\kappa}_{p,\sigma}(\omega^p)$ where
\begin{eqnarray*}
  {\tilde{\kappa}}_1 =m+ n(1- \rho)\left( \frac{1}{2} - \frac{1}{s} \right) +\rho n\left(\frac{1}{q}-\frac{1}{p} \right).
\end{eqnarray*} 

\noindent$({\rm{ii}})$ Let $1< q  \le 2 \le p \le q' <\infty$ and $\omega^q \in A_{q/r} \cap {RH}_{(p/q)(r'/p)'}$ with some $r \in [1 ,q)$. Then, for any $\kappa \in \R$ and $0<\sigma \le \infty$, $a(x,D)$ be a bounded operators from $B^{ \kappa + {\tilde{\kappa}}_2 }_{q,\sigma}(\omega^q)$ to $ B^{\kappa}_{p,\sigma}(\omega^p)$ where
\begin{eqnarray*}
  {\tilde{\kappa}}_2 =m+ n(1- \rho)\left( \frac{2}{r} - 1 \right) +\rho n\left(\frac{1}{q}-\frac{1}{p} \right).
\end{eqnarray*} 

\end{cor}

\begin{rem}
The Corollary \ref{pse Besov} contains the following known boundedness results of $a(x,D)$ with $a \in S^{m}_{\rho , \delta}$.

\noindent$({\rm{i}})$ By taking $\omega =1$, $p=q$ and suitable $s$ in neighborhood of $p$ in $({\rm{i}})$ of Corollary \ref{pse Besov}, we have the $L^p$-boundedness with $m<-n(1-\rho)|1/p-1/2|$ which was established by Fefferman {\rm \cite{F}}.

\noindent$({\rm{ii}})$ By taking $p=q$ and sufficiently large $s$ in $({\rm{i}})$ of Corollary \ref{pse Besov}, we have the $L^p(\omega)$-boundedness with $m=-n(1-\rho)/2$ and $\omega \in  A_{p/2} $ which was established by Chanillo and Torchinsky {\rm \cite{C-T}}.

\noindent$({\rm{iii}})$ By taking $r=1$ and $p=q$ in $({\rm{ii}})$ of Corollary \ref{pse Besov}, we have the $L^p(\omega)$-boundedness with $m=-n(1-\rho)$ and $\omega  \in A_{p}$ which was established by Michalowski, Rule and Staubach {\rm \cite{M-R-S}}.

\end{rem}

\begin{proof}
First, we assume $1 \le \sigma <\infty$. For any $ \ell \in  \Z$, there exists $b_{\ell} \in S^{m+\ell}_{\rho , \delta}$ so that $b(x,D)= {\lan D\ran }^\ell a(x,D)$ since $\delta <\rho$. From this, we have
\begin{eqnarray*}
{| \lan \phi_k * a(x,D)(\phi_j *f) , g \ran|} &=&  |\lan {\lan D \ran}^{-\ell} \phi_k * {\lan D \ran}^{\ell} a(x,D) (\phi_j * f) , g \ran | \\
&=&  |  \lan b_{\ell} (x,D) f ,  {\lan D  \ran }^{-\ell} \phi_k(-\cdot) *g \ran | \\
&\lesssim&  2^{j\ell+ j {\tilde{\kappa}}(p,q)} \liminf_{|R| \to \infty}  {\Lambda}^{\alpha}_{r(p,q),s(p,q)}((\phi_j*f)1_{Q_R} , {\lan D  \ran }^{-\ell}  \phi_k(-\cdot) *g ) \\
&\lesssim& 2^{j\ell+ j {\tilde{\kappa}}(p,q) } {||\phi_j *f||}_{L^q(\omega^q)} {|| {\lan D  \ran }^{-\ell}  \phi_k(-\cdot) *g  ||}_{L^{p'}(\omega^{-p'} )},
\end{eqnarray*}
where 
\[
({\tilde{\kappa}}(p,q),r(p,q),s(p,q))=\begin{cases}
({\tilde{\kappa}}_1,2,s') & 2<q \le p <\infty  \\
({\tilde{\kappa}}_2, r , r) & 1<q \le 2 \le p \le q' <\infty
\end{cases},
\]
and $\alpha= 1/p' + 1/q$. After this, we write ${\tilde{\kappa}}= {\tilde{\kappa}}(p,q)$. Now, we obtain
\begin{eqnarray*}
 {\lan D  \ran }^\ell \phi_k(-\cdot) *g  (x) \lesssim {||{\lan D\ran}^{-\ell} {\phi}_k||}_{L^1} Mg(x) \lesssim 2^{-k \ell }Mf(x).
\end{eqnarray*} 
Combining this and ${\omega}^{-p'} \in A_{p'}$, we obtain
\begin{eqnarray*}
{||\phi_k *a(x,D)(\phi_j *f) ||}_{L^p(\omega^p) } \lesssim 2^{-k \ell}  2^{j\ell+ j {\tilde{\kappa}} } {||\phi_j *f||}_{L^q(\omega^q)}.
\end{eqnarray*}
The Besov norm of $a(x,D)$ is handled by
\begin{eqnarray*}
&&I_1+I_2 \\
&:=& {\left( \sum_{k \ge 0} 2^{k \kappa \sigma} { \left( \sum_{ 0 \le j \le k} {|| \phi_k *a(x,D)(\phi_j *f) ||}_{L^p(\omega^p)}\right) }^{\sigma} \right)}^{1/\sigma} + {\left( \sum_{k \ge 0} 2^{k \kappa \sigma} { \left( \sum_{ k< j } {|| \phi_k *a(x,D)(\phi_j *f) ||}_{L^p(\omega^p)}\right) }^{\sigma} \right)}^{1/\sigma}.
\end{eqnarray*}
Our purpose is to control $I_1$ and $I_2$ by ${||f||}_{B^{\kappa +\tilde{\kappa} }_{q,\sigma }( \omega^q)}$. First, we give an estimation of $I_1$. From the observation above, we obtain
\begin{eqnarray*}
I_1&\lesssim&  {\left( \sum_{k \ge 0} 2^{k \kappa \sigma -k \ell \sigma } { \left( \sum_{ 0 \le j \le k} {2^{j \ell +j \tkappa}|| \phi_j *f||}_{L^q(\omega^q)}\right) }^{\sigma} \right)}^{1/\sigma} \\
&\le&  {\left( \sum_{k \ge 0} 2^{k \kappa \sigma -k \ell \sigma } { \left( \sum_{-k \le  j \le 0} {2^{(j+k) \ell +(j+k) \tkappa}|| \phi_{j+k} *f||}_{L^q(\omega^q)}\right) }^{\sigma} \right)}^{1/\sigma} \\
&\le&  \sum_{j \le 0} 2^{j \ell+ j \tkappa} {\left( \sum_{k \ge -j } 2^{k(\kappa + \tkappa)\sigma }  {|| \phi_{j+k} *f||}^{\sigma}_{L^q(\omega^q)}  \right)}^{1/\sigma} \\
&=&  \sum_{j \le 0} 2^{j \ell+ j \tkappa-j(\kappa+\tkappa)} {\left( \sum_{k \ge 0 } 2^{k(\kappa + \tkappa)\sigma }  {|| \phi_{k} *f||}^{\sigma}_{L^q(\omega^q)}  \right)}^{1/\sigma} \\
&\lesssim& {||f||}_{B^{\kappa +\tkappa}_{q,\sigma}(\omega^q)}
\end{eqnarray*}
by taking sufficiently large $\ell$. On the other hands, the same calculation gives 
\begin{eqnarray*}
I_2&\lesssim&  {\left( \sum_{k \ge 0} 2^{k \kappa \sigma -k \ell \sigma } { \left( \sum_{ k<j } {2^{j \ell +j \tkappa}|| \phi_j *f||}_{L^q(\omega^q)}\right) }^{\sigma} \right)}^{1/\sigma} \\
&\le&  {\left( \sum_{k \ge 0} 2^{k \kappa \sigma -k \ell \sigma } { \left( \sum_{ j>0 } {2^{(j+k) \ell +(j+k) \tkappa}|| \phi_{j+k} *f||}_{L^q(\omega^q)}\right) }^{\sigma} \right)}^{1/\sigma} \\
&\le&  \sum_{j > 0} 2^{j \ell+ j \tkappa} {\left( \sum_{k \ge 0 } 2^{k(\kappa + \tkappa)\sigma }  {|| \phi_{j+k} *f||}^{\sigma}_{L^q(\omega^q)}  \right)}^{1/\sigma} \\
&=&  \sum_{j > 0} 2^{j \ell+ j \tkappa-j(\kappa+\tkappa)} {\left( \sum_{k \ge 0 } 2^{k(\kappa + \tkappa)\sigma }  {|| \phi_{k} *f||}^{\sigma}_{L^q(\omega^q)}  \right)}^{1/\sigma} \\
&\lesssim& {||f||}_{B^{\kappa +\tkappa}_{q,\sigma}(\omega^q)}
\end{eqnarray*}
by taking $\ell \ll -1$. Hence, we complete the proof in the case of $1\le \sigma <\infty$. To complete the proof, we treat the case of $0<\sigma <1$. However, this case is proved in a same manner by using $\sigma$-triangle inequality on $\ell^{\sigma}$.
\end{proof}

%%%%%%%%%%%%%%%%%%%%%%%%%%%%% Dissipative  equation %%%%%%%%%%%%%%%%%%%%%%%%%%%%%%%%%%%%%%%%%%%%%%%%
\subsection{The special case of pseudodifferential operators}
For a given $-1\le \rho<1$, $U_{\rho}f$ denotes the solution of
\begin{eqnarray*}
\left\{
\begin{array}{ll}
i{\partial}_{t}u+{(-\Delta) }^{(1-\rho)/2} u=  0, \\
u(0)=f.\\
\end{array}
\right.
\end{eqnarray*} 
$U_{\rho}$ with $0\le \rho <1$ can be regarded as a pseudodifferential operators associated $S^{0}_{\rho,0}$, and therefore gives sparse bounds in Theorem \ref{sparseform1} and Theorem \ref{sparseform2}. However, we can improve the above results:
\begin{thm}
\label{sparseform3}
Let $1\le r \le 2 $ and $-1\le \rho <1$.

\noindent$({\rm{i}})$ Given $\rho \neq 0$, $ 1/r+1/2 < 1/ \alpha < 2/r$ and $f ,g \in \sh $, there exist the sequence of sparse families ${\{\Sp_j \}}_{j =0,1,\cdots}$ such that
\begin{eqnarray*}
|\lan U_{\rho} f(t) , g \ran | \lesssim t^{n(1/r+1/2-(1/\alpha-1))}  \liminf_{R \to \infty} \sum_{j \ge 0} 2^{j \kappa_4 } {\Lambda}^{\alpha}_{ \Sp_j ,r,r}( (\phi_j *f)1_{Q_R},g),
\end{eqnarray*}
where $\kappa_4= n(1- \rho)(1/r-1/2)+ \rho n (1/\alpha -1)$.

\noindent$({\rm{ii}})$ Given $\alpha \in \R $ such that
\begin{eqnarray*}
\frac{n+1}{rn} + \frac{n-1}{2n} < \frac{1}{\alpha} < \frac{2}{r},
\end{eqnarray*}
and $f ,g \in \sh $, there exist the sequence of sparse families ${\{\Sp_j \}}_{j =0,1,\cdots}$ such that
\begin{eqnarray*}
|\lan U_{0} f(t) , g \ran | \lesssim t^{(n+1)(1/r-1/2)-n(1/\alpha-1)}  \liminf_{R \to \infty} \sum_{j \ge 0} 2^{j \kappa_5 } {\Lambda}^{\alpha}_{ \Sp_j ,r,r}( (\phi_j *f)1_{Q_R},g),
\end{eqnarray*}
where $\kappa_5= (n+1)(1/r-1/2)$.
\end{thm}

\begin{proof}
\noindent$({\rm{i}})$ It suffices to prove the pointwise estimate
\begin{eqnarray*}
 M_{U_{\rho,j}, \infty} f(t,x) \lesssim t^{-n(\gamma -1/2)}   2^{jn(1-\rho)/2 +jn \rho \gamma } M^{\gamma} f(x)
\end{eqnarray*}
for any $1/2< \gamma< 1$ where $U_{\rho,j}f=U_{\rho}(\phi_j*f)$. Take any cube $Q$ and any $x, z \in Q$. First, we consider the case $ j \ge 1$ and $2^{-j(1-\rho)}  \le t$. We integrate by parts $N \in \N$ times to obtain
\begin{eqnarray*}
 |U_{\rho,j} (f1_{{(3Q)}^c } )(t,z)| &\lesssim& \int_{ {(3Q)}^c} {(1+2^{2j\rho N } {t}^{-2N}{|z-y|}^{2N})}^{-1} |f(y)| \left| \int e^{i(z-y)\xi } { (1+2^{2j\rho N} {t}^{-2N} \Delta^N )} (e^{it{|\xi|}^{1-\rho} } \hphi_j(\xi))d\xi \right| dy  \\
&\lesssim& 2^{-j\rho n (1-\gamma) }{t}^{n(1-\gamma)} M^{\gamma}f(x) \left| \int e^{i(z-y)\xi } { (1+2^{2j\rho N} {t}^{-2N} \Delta^N )} (e^{it{|\xi|}^{1-\rho} } \hphi_j(\xi))d\xi dy \right|  . \\
\end{eqnarray*}
To obtain desired pointwise estimate,  we need to prove
\begin{eqnarray*}
\sup_{w \in \Rn} \left| \int e^{i w \xi } { (1+2^{2j\rho N} {t}^{-2N} \Delta^N )} (e^{it{|\xi|}^{1-\rho} } \phi_j(\xi))d\xi dy \right| \lesssim 2^{jn(1+\rho)/2 } t^{-n/2}.
\end{eqnarray*}
By the Leibniz formula, we have 
\begin{eqnarray*}
 \Delta^N ((e^{it{|\xi|}^{1-\rho} } \hphi_j(\xi)) &=& \sum_{|\alpha|=2N } \sum_{\beta \le \alpha} (\partial^{\alpha -\beta} e^{it{|\xi|}^{1-\rho} })(\partial^{\beta} \hphi_j(\xi)) \\
 &=&e^{it{|\xi|}^{1-\rho} } \sum_{|\alpha|=2N } \sum_{\beta \le \alpha} (P_{\alpha ,\beta}(\xi) )(\partial^{\beta} \hphi_j(\xi) ), 
\end{eqnarray*}
where $P_{\alpha, \beta}$ denotes the functions such that
\begin{eqnarray*}
{||\partial^{\sigma }P_{\alpha, \beta}(2^j \cdot)||}_{L^{\infty} } \lesssim{ (2^{-j}+t{2}^{-j\rho})}^{2N- |\beta|} 
\end{eqnarray*} 
on support of $\hat{\psi}$ for any $\sigma \in \N^n$. By using Littman's lemma, we have
\begin{eqnarray*}
&&\left| \int e^{i w \xi } { (1+2^{2j\rho N} {t}^{-2N} \Delta^N )} (e^{it{|\xi|}^{1-\rho} } \hphi_j(\xi))d\xi dy \right|\\
&\lesssim& \left| \int e^{iw \xi +it{|\xi|}^{1-\rho} }\hphi_j(\xi)d\xi \right| +  2^{2j\rho N} {t}^{-2N} \sum_{|\alpha|=2N } \sum_{\beta \le \alpha} \left| \int e^{iw\xi+ it{|\xi|}^{1-\rho} }P_{\alpha.\beta} (\xi) \partial^{\beta } \hphi_j(\xi)d\xi \right| \\
&\lesssim& 2^{jn} \left| \int e^{iw \xi +it2^{j(1-\rho)}{|\xi|}^{1-\rho} }\hat{\psi}(\xi)d\xi \right| +  2^{jn+ 2j\rho N} {t}^{-2N} \sum_{|\alpha|=2N }\left|  \sum_{\beta \le \alpha} 2^{- j |\beta| }  \int e^{iw\xi+ it2^{j(1-\rho)} {|\xi|}^{1-\rho} }P_{\alpha.\beta} (2^j \xi) \partial^{\beta} \hat{\psi}(\xi)d\xi \right|  \\
&\lesssim&2^{jn(1+\rho)/2 } t^{-n/2}+  2^{jn(1+\rho)/2+ 2j\rho N} {t}^{-n/2-2N}  \sum_{|\alpha|=2N } \sum_{\beta \le \alpha} 2^{- j |\beta|}  {({2}^{-j}+t{2}^{-j\rho})}^{2N- |\beta|} \\
&\lesssim& 2^{jn(1+\rho)/2 } t^{-n/2}.
\end{eqnarray*}
Here, the last inequality follows from
\begin{eqnarray*}
\sum_{|\alpha|=2N } \sum_{\beta \le \alpha} 2^{- j |\beta| }  {({2}^{-j}+t{2}^{-j\rho})}^{2N- |\beta|} \lesssim t^{2N } 2^{-2j \rho N} \sum_{|\alpha|=2N } \sum_{\beta \le \alpha}  t^{- |\beta|}  2^{-j |\beta| + j\rho  |\beta| }\lesssim t^{2N } 2^{-2j \rho N} .
\end{eqnarray*}
The case of $j \ge 1$ and $2^{j(1-\rho) } \le t^{-1}$ is obtained from
\begin{eqnarray*}
 |U_{\rho,j} (f1_{{(3Q)}^c } )(t,z)| &\lesssim& \int_{ {(3Q)}^c} {(1+2^{2j N } {|z-y|}^{2N})}^{-1} |f(y)| \left| \int e^{i(z-y)\xi } { (1+2^{2j N}  \Delta^N )} (e^{it{|\xi|}^{1-\rho} } \hphi_j(\xi))d\xi\right| dy  \\
&\lesssim& 2^{-j  n (1-\gamma) } M^{\gamma}f(x) \left| \int e^{i(z-y)\xi } { (1+2^{2jN}  \Delta^N )} (e^{it{|\xi|}^{1-\rho} } \hphi_j(\xi))d\xi dy \right|  \\
&\lesssim&2^{jn \gamma} M^{\gamma}f(x) \\
&\le&t^{-n(\gamma -1/2)}   2^{jn(1-\rho)/2 +jn \rho \gamma } M^{\gamma} f(x).
\end{eqnarray*}
Here, we use the condition $\gamma  \ge 1/2$ to obtain
\begin{eqnarray*}
2^{jn\gamma} = 2^{jn(1-\rho)/2 +jn \rho \gamma } 2^{jn(1-\rho)(\gamma -1/2)} \le t^{-n(\gamma -1/2)}   2^{jn(1-\rho)/2 +jn \rho \gamma }.
\end{eqnarray*}
When $j =0$, we recall $\phi_0 =\sum_{\ell\le 0} \psi_{\ell}$ and obtain
\begin{eqnarray*}
 |U_{\rho,0} (f1_{{(3Q)}^c } )(t,z)| &\le& \sum_{\ell \le 0} \left| \int e^{i(z-y)\xi }  {\hat{\psi}}_{\ell} (\xi)f(y)1_{{(3Q)}^c} (y) dyd\xi \right|\\
&\lesssim& \sum_{ \ell \le 0} \int_{ {(3Q)}^c} {(1+2^{2\ell N } {\tau}^{-2N}{|z-y|}^{2N})}^{-1} |f(y)| \int e^{i(z-y)\xi } { (1+2^{2\ell N} {\tau}^{-2N} \Delta^N )} (e^{it{|\xi|}^{1-\rho} } {\hat{\psi}}_{\ell}(\xi))d\xi dy  \\
&\lesssim& \sum_{\ell \le 0} 2^{-\ell n (1-\gamma) }{\tau}^{n(1-\gamma)} M^{\gamma}f(x) \left| \int e^{i(z-y)\xi } { (1+2^{2\ell N} {\tau}^{-2N} \Delta^N )} (e^{it{|\xi|}^{1-\rho} } {\hat{\psi}}_{\ell}(\xi))d\xi dy \right|  , \\
\end{eqnarray*}
where $\tau= \max\{1,t\}$. Since ${||\partial^{\sigma} P_{\alpha, \beta} (2^{\ell} \cdot ) ||}_{L^{\infty} } \lesssim {\tau}^{|\alpha| -|\beta|} 2^{-\ell(|\alpha|- |\beta|)}$, one has
\begin{eqnarray*}
|U_{\rho,0} (f1_{{(3Q)}^c } )(t,z)| &\lesssim&  \left(\sum_{ \ell \le 0} 2^{\ell n (\gamma -1/2)} \right) {\tau}^{n(1-\gamma)-n/2} M^{\gamma}f(x)  \\
&\lesssim& t^{-n(\gamma -1/2)}  M^{\gamma} f(x).
\end{eqnarray*}

 \ 
 
\noindent$({\rm{ii}})$ It suffices to prove the pointwise estimate
\begin{eqnarray*}
 M_{U_{0,j}, \infty} f(t,x) \lesssim t^{-n(\gamma -1/2)+1/2}  2^{j(n+1)/2 }  M^{\gamma} f(x)
\end{eqnarray*}
for any $  (n+1)/2n < \gamma< 1$. Take any cube $Q$ and any $x, z \in Q$. First, we consider the case $ j \ge 1$ and $2^{-j}  \le t$. We integrate by parts $N \in \N$ times to obtain
\begin{eqnarray*}
 |U_{0,j} (f1_{{(3Q)}^c } )(t,z)|  \lesssim {t}^{n(1-\gamma)} M^{\gamma}f(x) \left| \int e^{i(z-y)\xi } { (1+ {t}^{-2N} \Delta^N )} (e^{it{|\xi|} } \hphi_j(\xi))d\xi dy \right|.  \\
\end{eqnarray*}
By using Littman's lemma, we have
\begin{eqnarray*}
&&\left| \int e^{i w \xi } { (1+{t}^{-2N} \Delta^N )} (e^{it{|\xi|} } \hphi_j(\xi))d\xi dy \right|\\
&\lesssim& \left| \int e^{iw \xi +it{|\xi| } }\hphi_j(\xi)d\xi \right| +   {t}^{-2N} \sum_{|\alpha|=2N } \sum_{\beta \le \alpha} \left| \int e^{iw\xi+ it{|\xi|}  }P_{\alpha.\beta} (\xi) \partial^{\beta } \hphi_j(\xi)d\xi \right| \\
&\lesssim& 2^{jn} \left| \int e^{iw \xi +it2^j {|\xi|} }\hat{\psi}(\xi)d\xi \right| +  2^{jn} {t}^{-2N} \sum_{|\alpha|=2N }\left|  \sum_{\beta \le \alpha} 2^{- j |\beta| }  \int e^{iw\xi+ it2^j {|\xi|} }P_{\alpha.\beta} (2^j \xi) \partial^{\beta} \hat{\psi}(\xi)d\xi \right|  \\
&\lesssim&2^{j(n+1)/2 } t^{-(n-1)/2}+  2^{j(n+1)/2} {t}^{-(n-1)/2-2N}  \sum_{|\alpha|=2N } \sum_{\beta \le \alpha} 2^{- j |\beta|}  {({2}^{-j}+t)}^{2N- |\beta|} \\
&\lesssim& 2^{j(n+1)/2 } t^{-(n-1)/2}
\end{eqnarray*}
for any $ w \in \Rn$. The case of $j \ge 1$ and $2^{j } \le t^{-1}$ is obtained from
\begin{eqnarray*}
 |U_{0,j} (f1_{{(3Q)}^c } )(t,z)| &\lesssim& \int_{ {(3Q)}^c} {(1+2^{2j N } {|z-y|}^{2N})}^{-1} |f(y)| \left| \int e^{i(z-y)\xi } { (1+2^{2j N}  \Delta^N )} (e^{it{|\xi|} } \hphi_j(\xi))d\xi\right| dy  \\
&\lesssim& 2^{-j  n (1-\gamma) } M^{\gamma}f(x) \left| \int e^{i(z-y)\xi } { (1+2^{2jN}  \Delta^N )} (e^{it{|\xi|} } \hphi_j(\xi))d\xi dy \right|  \\
&\lesssim&2^{jn \gamma} M^{\gamma}f(x) \\
&\le&  t^{-n(\gamma -1/2)+1/2}  2^{j(n+1)/2 }  M^{\gamma} f(x).
\end{eqnarray*}
Here, we use the condition $\gamma  \ge (n+1)/2n$ to obtain
\begin{eqnarray*}
2^{jn\gamma} =2^{-j(n+1)/2 +jn \gamma} 2^{j(n+1)/2  }  \le  t^{-n(\gamma -1/2)+1/2}  2^{j(n+1)/2 }   .
\end{eqnarray*}
When $j =0$, we recall $\phi_0 =\sum_{\ell\le 0} \psi_{\ell} $ and obtain
\begin{eqnarray*}
 |U_{0,0} (f1_{{(3Q)}^c } )(t,z)| &\le& \sum_{\ell \le 0} \left| \int e^{i(z-y)\xi } a(z,\xi) {\hat{\psi}}_{\ell} (\xi)f(y)1_{{(3Q)}^c} (y) dyd\xi \right|\\
&\lesssim& \sum_{ \ell \le 0} \int_{ {(3Q)}^c} {(1+2^{2\ell N } {\tau}^{-2N}{|z-y|}^{2N})}^{-1} |f(y)| \int e^{i(z-y)\xi } { (1+2^{2\ell N} {\tau}^{-2N} \Delta^N )} (e^{it{|\xi|} } {\hat{\psi}}_{\ell}(\xi))d\xi dy  \\
&\lesssim& \sum_{\ell \le 0} 2^{-\ell n (1-\gamma) }{\tau}^{n(1-\gamma)} M^{\gamma}f(x) \left| \int e^{i(z-y)\xi } { (1+2^{2\ell N} {\tau}^{-2N} \Delta^N )} (e^{it{|\xi|}  } {\hat{\psi}}_{\ell}(\xi))d\xi dy \right|  , \\
\end{eqnarray*}
where $\tau= \max\{1,t\}$. Since ${||\partial^{\sigma} P_{\alpha, \beta} (2^{\ell} \cdot ) ||}_{L^{\infty} } \lesssim {\tau}^{|\alpha| -|\beta|} 2^{-\ell(|\alpha|-|\beta|  )}$, one has
\begin{eqnarray*}
|U_{0,0} (f1_{{(3Q)}^c } )(t,z)| &\lesssim&  \left(\sum_{ \ell \le 0} 2^{\ell n (\gamma -(n+1)/2n)} \right) {\tau}^{n(1-\gamma)-(n+1)/2} M^{\gamma}f(x)  \\
&\lesssim&   t^{-n(\gamma -1/2)+1/2}M^{\gamma} f(x).
\end{eqnarray*}

\end{proof}

Theorem \ref{sparseform3} and Proposition \ref{weight} give the boundness of $U_{\rho}$ on weighted Besov spaces.
\begin{cor}
\label{weight sol}
Let $ -1 \le \rho <1$, $ 1<q \le 2 \le p \le q' <\infty$ and $ {\omega}^q \in A_{q/r} \cap RH_{(p/q)(r'/p)'}$ with some $r\in[1,q)$. 
\noindent$({\rm{i}})$ If $\rho \neq 0$ and 
\begin{eqnarray*}
\frac{1}{r} -\frac{1}{2} \le \frac{1}{q} -\frac{1}{p},
\end{eqnarray*}
then for any $\kappa  \in \R$ and $0<\sigma \le \infty$, $U_{\rho}(t)$ be a bounded operators from $B^{\kappa+\tkappa_4}_{q,\sigma}(\omega^q)$ to $B^{\kappa}_{p,\sigma}(\omega^p)$ where 
\begin{eqnarray*}
\tkappa_4 =n(1- \rho)\left( \frac{1}{r}- \frac{1}{2} \right)+ \rho n \left( \frac{1}{q} -\frac{1}{p}\right).
\end{eqnarray*}
Furthermore, we have
\begin{eqnarray*}
{||U_{\rho}(t)||}_{B^{\kappa }_{q,\sigma}(\omega^q) \to B^{\kappa +\tkappa_4 }_{p,\sigma}(\omega^p)} \lesssim t^{-n((1/q-1/p)-(1/r -1/2))}  {({[{\omega}^q]}_{A_{q/r}} {[ {\omega}^q] }_{{RH}_{(p/q)( r'/p)' }})}^{\delta} ,
\end{eqnarray*}
with $\delta $ in Proposition \ref{weight}.

\noindent$({\rm{ii}})$ If
\begin{eqnarray*}
\frac{n+1}{n}\left( \frac{1}{r} -\frac{1}{2} \right) \le \frac{1}{q} -\frac{1}{p},
\end{eqnarray*}
then for any $\kappa  \in \R$ and $0<\sigma \le \infty$, $U_{\rho}(t)$ be a bounded operators from $B^{\kappa+\tkappa_5}_{q,\sigma}(\omega^q)$ to $B^{\kappa}_{p,\sigma}(\omega^p)$ where
\begin{eqnarray*}
\tkappa_5 =(n+1)\left( \frac{1}{r}- \frac{1}{2} \right).
\end{eqnarray*}
Furthermore, we have
\begin{eqnarray*}
{||U_{\rho}(t)||}_{B^{\kappa }_{q,\sigma}(\omega^q) \to B^{\kappa +\tkappa_5 }_{p,\sigma}(\omega^p)} \lesssim t^{-\{n(1/q-1/p) -(n+1)(1/r-1/2) \}}   {({[{\omega}^q]}_{A_{q/r}} {[ {\omega}^q] }_{{RH}_{(p/q)( r'/p)' }})}^{\delta} .
\end{eqnarray*}
\end{cor}
 
  \ 
  
%%%%%%%%%%%%%%%%%%%%%%%%%%sharpness%%%%%%%%%%%%%%%%%%%%%%%%%%%%%%%%%%%%%%%%
\subsection{A sharpness of weighted boundedness of pseudodifferential operators}

In previous sections and subsections, we obtain some weighted inequalities for pseudodifferential operators and the time evolution $U_{\rho}(t)$ of dispersive equations. In this subsection, we insure a sharpness of some of these inequalities as follows:
\begin{pro}
\label{sharp}
Let $1< q \le p \le q' <\infty$ and $\gamma \in [1,\infty)$, and $a(\xi)= e^{i {|\xi|}^{1-\rho}} {|\xi|}^m$ with $m \in \R$ and $0< \rho \le 1$. If we have $L^q( {|\cdot|}^{qs})$-$L^p({|\cdot|}^{ps})$ boundedness of $a(D)$ for any $s\in (-n/\gamma,0)$, then we have
\begin{eqnarray*}
m \le -n(1-\rho) \left(\frac{1}{2}-\frac{1}{p}\right)-\rho n\left(\frac{1}{q}-\frac{1}{p}\right)-\frac{n(1-\rho)}{\gamma}.
\end{eqnarray*}
In particular, if we have $L^q( {\omega}^{q})$-$L^p({\omega}^{p})$ boundedness of $a(D)$ with any $\omega ^q \in  RH_{(p/q)(r'/p)'}$ for some $ r \in [1,q)$, then we have
\begin{eqnarray*}
m \le -n(1-\rho) \left(\frac{1}{r}-\frac{1}{2}\right)-\rho n\left(\frac{1}{q}-\frac{1}{p}\right).
\end{eqnarray*}
\end{pro} 

\begin{proof}
Our assumption gives
\begin{eqnarray*}
|\lan a(D)f,g \ran| \lesssim {||f||}_{L^q({|\cdot|}^{qs})} {||g||}_{L^{p}{({|\cdot|}^{-p's}) }}
\end{eqnarray*}
for any $s \in (-n/ \gamma,0)$. We take a nonnegative function $\phi \in C^{\infty}_{0}$ such that $supp \ \phi \subset \{ 1/4 \le |\xi| \le 2\}$ and $\phi =1 \ on \ \{1/2 \le |\xi| \le1\}$, and let
\begin{eqnarray*}
\hat{f} (\xi)=  e^{-i{|\xi|}^{1-\rho} } \phi(\xi/R) ,
\end{eqnarray*}
and 
\begin{eqnarray*}
\check{g} (\xi)=\phi(\xi/R)
\end{eqnarray*}
for any $R>0$. Then, we have
\begin{eqnarray*}
|\lan a(D)f ,g \ran | &=& \left|\int {|\xi|}^{m} \phi(\xi/R) \phi(\xi/R) d\xi\right| \\
&\sim&R^{m+n}.
\end{eqnarray*}
On the other hands, we have
\begin{eqnarray*}
|f(x)| \lesssim \min \{R^{n(1+\rho)/2}, R^{n(1+\rho)/2-2\rho N} {|x|}^{-2N} \}
\end{eqnarray*}
for any $N \in \N$. In fact, Littman's lemma gives
\begin{eqnarray*}
|f(x)| &=& \left| \int  e^{ix\xi -i {|\xi|}^{1-\rho}  } \phi(\xi /R) d\xi \right| \\
&\le& R^n\sup_{z} \left| \int  e^{iz\xi +iR^{1-\rho} {|\xi|}^{1-\rho} } \phi(\xi) d\xi \right| \\
&\lesssim& R^n R^{-n(1-\rho)/2}\\
&=&R^{n(1+\rho)/2}.
\end{eqnarray*}
As  for second estimates, we have
\begin{eqnarray*}
|f(x)| &=& {|x|}^{-2N} \left| \int \Delta^{N}_{\xi} e^{ix\xi -i {|\xi|}^{1-\rho}  } \phi(\xi /R) d\xi \right| \\
&=& {|x|}^{-2N} \left| \int  e^{ix\xi}  \Delta^{N}_{\xi} (e^{ -i {|\xi|}^{1-\rho}  } \phi(\xi /R) )d\xi \right|\\
&\le& {|x|}^{-2N}  \sum_{|\alpha|=2N } \sum_{\beta \le \alpha} R^{-|\beta|}\left| \int  e^{ix\xi - it{|\xi|}^{1-\rho} } (P_{\alpha ,\beta}(\xi) )(\partial^{\beta} \phi)(\xi/R)d\xi \right|\\
\end{eqnarray*}
where $P_{\alpha, \beta}$ denotes the functions such that
\begin{eqnarray*}
{||\partial^{\sigma }P_{\alpha, \beta}(R \cdot)||}_{L^{\infty} } \lesssim{ R}^{-2\rho N+\rho |\beta|} 
\end{eqnarray*} 
on support of $\phi$ for any $\sigma \in \N^n$. From this and Littman's lemma, we obtain desired estimate. Therefore, we have
\begin{eqnarray*}
{||f||}_{L^q({|\cdot|}^{qs})} &\le& {\left( \int_{|x|\le R^{-\rho}  }{|f(x)|}^q {|x|}^{qs} dx \right)}^{1/q} +  {\left( \int_{|x|\ge R^{-\rho}  }{|f(x)|}^q {|x|}^{qs} dx \right)}^{1/q} \\
&\le&R^{ n(1+\rho)/2}  {\left( \int_{|x|\le R^{-\rho}  } {|x|}^{qs} dx \right)}^{1/q} +  R^{n(1+\rho) /2 -2\rho N} {\left( \int_{|x|\ge R^{-\rho}  } {|x|}^{qs-2qN} dx \right)}^{1/q}  \\
&\lesssim& R^{ n(1+\rho)/2-\rho n/q -\rho s},\\
\end{eqnarray*}
and
\begin{eqnarray*}
{||g||}_{L^{p'}{( {|\cdot|}^{-p's}) }} &=&R^n {\left(\int {|\hat{\phi}(Rx)|}^{p'} {|x|}^{-p's } dx \right)}^{1/p'}  \\
&\lesssim& R^{n-n/p'+s}\\
&=&R^{n/p+s}.
\end{eqnarray*}
From these observations, we obtain
\begin{eqnarray*}
R^{m+n} &\lesssim&  R^{  n(1+\rho)/2-\rho n/q -\rho s }R^{n/p+s}\\
R^{m} &\lesssim& R^{ -n(1/2-1/p)+n\rho(1/2-1/q)+s(1-\rho)} \\
R^{m} &\lesssim& R^{ -n(1-\rho)(1/2-1/p)-\rho n(1/q-1/p)+s(1-\rho)} \\
m&\le& -n(1-\rho)(1/2-1/p)-\rho n(1/q-1/p)-n(1-\rho)/\gamma
\end{eqnarray*}
Here, we take the infimum all over the $s\in(-n/\gamma,0)$ to obtain the final inequality. In particular, we have ${| \cdot |}^{qs} \in  RH_{(p/q)(r'/p)'}$ with $s \in (-n/p(r'/p)',0)$, that means
\begin{eqnarray*}
m &\le& -n(1-\rho)(1/2-1/p)-\rho n(1/q-1/p)-n(1-\rho)/p(r'/p)' \\
&=& -n(1-\rho)(1/2-1/p)-\rho n(1/q-1/p)-n(1-\rho)(1-p/r')/p\\
&=& -n(1-\rho)(1/r-1/2)-\rho n(1/q-1/p).
\end{eqnarray*}
by taking $\gamma = p(r'/p)'$.
\end{proof}
\begin{rem}
Since $e^{i {|\xi|}^{1-\rho}} {|\xi|}^m \notin S^{m}_{\rho,0}$, Proposition \ref{sharp} cannot be applied to the pseudodifferential operators associated with  symbols belonging to the H\"{o}rmander class directly. However, by the same proof of the proposition, it holds with $a\in S^{m}_{\rho,0}$ such that $a(\xi) =e^{i {|\xi|}^{1-\rho}} {|\xi|}^m$ for any $|\xi| >1$, that means a sharpness of weighted inequalities in Theorem \ref{Coifman-Fefferman} and $({\rm{i}})$ of Corollary \ref{pse Besov}.
\end{rem}
%%%%%%%%%%%%%%%%%%%%%%%%%%%%% Appendix %%%%%%%%%%%%%%%%%%%%%%%%%%%%%%%%%%
\appendix
\section{Appendix A}
To see the proof of Corollary \ref{pse Besov}, the operator norms of $a(x,D)$ on weighted Besov spaces are controlled by 
\begin{eqnarray*}
{({[{\omega}^q]}_{A_{q/r}} {[ {\omega}^q] }_{{RH}_{(p/q)( r'/p)' }})}^{\delta} { [\omega^{-p'}]}_{A_{p'}} .
\end{eqnarray*}
However, we can eliminate the factor ${ [\omega^{-p'}]}_{A_{p'}}$ by having the sparse form bounds $\phi_k *a(x,D)(\phi_j *\cdot )$ directly.
\begin{pro}
\noindent$({\rm{i}})$ Let $2 \le s \le \infty$ and $ 2/3 < \alpha \le1$, and $ a \in S^{m}_{\rho , \delta}$ with $m \le 0$, $0\le \delta < \rho \le 1$. Then for any $f ,g \in \sh$ and $j,k \in \Z_{\ge 0}$, there exists the sparse family $ \Sp $ such that
\begin{eqnarray*}
|\lan \phi_k *a(x,D)(\phi_j *f ), g \ran | \lesssim 2^{-k\ell}2^{j\ell+  j \kappa_1} \liminf_{ R\to \infty}  {\Lambda}^{\alpha}_{ \Sp ,2,s'}( (\phi_j *f )1_{Q_R},g).
\end{eqnarray*}
$({\rm{ii}})$ Let $ 2 \le s \le \infty$ and $s'/2 < \alpha \le 1$, and $ a \in S^{m}_{\rho , \delta}$ with $m \le 0$, $0\le \delta < \rho \le 1$. Then for any $f ,g \in \sh $, there exists the sparse family $ \Sp $ such that
\begin{eqnarray*}
|\lan \phi_k * a(x,D)(\phi_j*f) , g \ran | \lesssim 2^{-k \ell} 2^{j \ell +j \kappa_2}  \liminf_{R \to \infty}  {\Lambda}^{\alpha}_{ \Sp ,s',s'}(( \phi_j *f) 1_{Q_R}  ,g).
\end{eqnarray*}
\noindent$({\rm{iii}})$ Let $ 1\le s' \le r \le 2 \le s \le \infty$ and $ a \in S^{m}_{\rho , \delta}$ with $m \le 0$, $0\le \delta < \rho \le 1$. Then for any $f ,g \in \sh$, there exists the sparse family $\Sp $ such that
\begin{eqnarray*}
|\lan \phi_k * a(x,D)(\phi_j*f) , g \ran | \lesssim 2^{-k \ell} 2^{j \ell +j \kappa_3}  \liminf_{R \to \infty}  {\Lambda}_{ \Sp ,r ,s'}(( \phi_j *f) 1_{Q_R}  ,g).
\end{eqnarray*}
\end{pro}
\begin{proof}
We put
\begin{eqnarray*}
T_{j,k}f   := \phi_k *a(x,D)(\phi_j*f).
\end{eqnarray*}
Here, we remark that 
\begin{eqnarray*}
T_{j,k}   f &=& ({\lan D\ran }^{-\ell} \phi_k) *({\lan D \ran}^{\ell} a(x,D)(\phi_j*f)) \\
&=& ({\lan D\ran }^{-\ell} \phi_k) * b_{\ell }(x,D) (\phi_j *f),
\end{eqnarray*}
with some $b_{\ell} \in S^{m+\ell}_{\rho ,\delta}$. For any cube $Q$ and $x \ \in Q$, one has
\begin{eqnarray*}
&& {|| T_{j ,k}(f 1_{{(3Q)}^c})||}_{L^{\infty }(Q)} \\
&\le&   {|| ({\lan D\ran }^{-\ell} \phi_k) * [1_{{(2Q)}^c}b_{\ell }(x,D) (\phi_j *(f1_{{(3Q)}^c}))||}_{L^{\infty }(Q)} +2^{-k\ell} {||b_{\ell }(x,D) (\phi_j *(f1_{{(3Q)}^c}))||}_{L^{\infty }(2Q)} \\
&=:& f_0(x) + f_1(x) ,
\end{eqnarray*}
where
\begin{eqnarray*}
f_0(x)&:=&\sup_{Q\in x} {|| ({\lan D\ran }^{-\ell} \phi_k) * [1_{{(2Q)}^c}b_{\ell }(x,D) (\phi_j *(f1_{{(3Q)}^c}))||}_{L^{\infty}(Q)} ,\\
f_1(x)&:=&\sup_{Q \ni x} 2^{-k\ell}  {|| b_{\ell }(x,D) (\phi_j *(f1_{{(3Q)}^c}))||}_{L^{ \infty }(2Q)} .
\end{eqnarray*}
\noindent$({\rm{i}})$ Now, we have
\begin{eqnarray*}
f_0(x) &\lesssim& 2^{-k \ell} M [b_{\ell }(x,D) (\phi_j *(f1_{{(3Q)}^c}))](x) \\
&\lesssim& 2^{-k \ell} 2^{jm +j \ell +jn(1-\rho(1-\gamma))/2 } M M^{\gamma}_{ 2} f(x)
\end{eqnarray*}
for any $0 \le \gamma <1$. By using the argument in the proof of Theorem \ref{sparseform1}, it is not hard to see the
\begin{eqnarray*}
f_1(x)  \lesssim 2^{-k \ell} 2^{jm +j \ell +jn(1-\rho(1-\gamma))/2 }  M^{\gamma}_{ 2} f(x).
\end{eqnarray*}
Therefore, we obtain
\begin{eqnarray*}
{||M_{T_{j,k},\infty} f ||}_{L^2 \to L^{p_0 ,\infty}} \lesssim 2^{-k \ell} 2^{jm +j \ell +jn(1-\rho)/2 +jn\rho (1/2 -1/p_0) } 
\end{eqnarray*}
for any $p_0 \ge 2$. On the other hands, we have
\begin{eqnarray*}
{|| T_{j,k} ||}_{L^2 \to L^{p_1}} \lesssim 2^{-k\ell} 2^{jm+j \ell +jn(1/2-1/p)} \ \ and  \ \ {|| M_{T_{j,k} ,p_1} ||}_{L^2 \to L^{p_1,\infty}}\lesssim 2^{-k\ell} 2^{jm+j \ell +jn(1/2-1/p_1)}
\end{eqnarray*}
for any $p_1 \ge2 $. By interpolating them, we have desired sparse bounds.

\noindent$({\rm{ii}})$, $({\rm{iii}})$ It suffices to prove the pointwise estimate
\begin{eqnarray*}
f_0(x) +f_1(x) \lesssim  2^{-k \ell} 2^{jm +j \ell +jn(1-\rho(1-\gamma))}  M^{\gamma} f(x).
\end{eqnarray*}
We just handle the $f_0(x)$ since the estimate of $f_1(x)$ be obtained immediately from the proof of Theorem \ref{sparseform2}. For any $N \in \N$ and $h \in L^1$, we have
\begin{eqnarray*}
&& ({\lan D\ran }^{-\ell} \phi_k) * h (z)\lesssim 2^{ -k \ell +kn }\int {(1+ 2^{2k  N} {|z-y|}^{2N} )  }^{-1} |h(y)| dy,  \\
&&| b_{\ell }(x,D) (\phi_j *(f1_{{(3Q)}^c}))(y) | \lesssim 2^{j \ell + jm + jn } \int {(1+ 2^{2j \rho N} {|y-w|}^{2N} )  }^{-1} |f(w)|  1_{ {(3Q)}^c }(w)dw.
\end{eqnarray*}
Hence, we obtain
\begin{eqnarray*}
f_0(x)  \lesssim  2^{ -k \ell +kn }2^{j \ell + jm + jn}  \sup_{Q \ni x}{ ||\Phi * (|f| 1_{ {(3Q)}^c} ) ||}_{L^{\infty} (Q) },
\end{eqnarray*}
where $\Phi$ denotes the radial function
\begin{eqnarray*}
\Phi (z) 
=\int \frac{1}{1+2^{2j\rho N }{|z-y| }^{2N} } \cdot \frac{1}{1+2^{2k N }{|y| }^{2N} } dy.
\end{eqnarray*}
To complete the proof, we decompose the integral region:
\begin{eqnarray*}
\Phi (z)&=& \int_{2|y| <|z|} +\int_{2|z| \le |y|} +\int_{|z|/2 \le |y| <2|z|} .\\
\end{eqnarray*}
Since $ |z-y | \gtrsim|z|$ under the $2|y| < |z| $ or $ 2|z| \le |y|$, one has
\begin{eqnarray*}
\int_{2|y| <|z|} + \int_{2|z| \le |y|}  \lesssim \frac{2^{-kn}}{ 1+2^{2j \rho N} {|z|}^{2N} }.
\end{eqnarray*}
Furthermore, it is not hard to see that
\begin{eqnarray*}
\int_{ |z|/2 \le |y| <2|z|} \lesssim \min \left\{  \frac{ 2^{-j \rho n } }{ 1+2^{2kN} {|z|}^{2N}  } , \frac{ {|z|}^n  }{ 1+2^{2kN} {|z|}^{2N}  } \right\}.
\end{eqnarray*}
From them, for any $k \le j \rho$, we have
\begin{eqnarray*}
f_0(x) &\lesssim& 2^{-k\ell} 2^{j \ell + jm +jn(1-\rho) }  (2^{j \rho n\gamma} + 2^{kn\gamma} ) M^{\gamma} f(x) \\
&\le&2^{-k \ell} 2^{jm +j \ell +jn(1-\rho(1-\gamma))}  M^{\gamma} f(x).
\end{eqnarray*}
We assume $ k> j \rho$. Then, we have
\begin{eqnarray*}
\sup_{z \in Q} \int_{{(3Q)}^c} \frac{ {|z-w|}^n  }{ 1+2^{2kN} {|z-w|}^{2N}  } |f(w)| dw &\lesssim& \int \frac{ {|x-w|}^n  }{ 1+2^{2kN} {|x-w|}^{2N}  } |f(w)| dw \\
 &\le& \sum_{i  \in \Z}  \int_{|x-w| \sim 2^{-k} 2^i} \frac{ {|x-w|}^n  }{ 1+2^{2kN} {|x-w|}^{2N}  } |f(w)| dw \\
 & \lesssim& 2^{-kn} 2^{-kn(1-\gamma)} M^{\gamma}f(x)  \\
  & \lesssim& 2^{-kn} 2^{-j \rho n(1-\gamma)} M^{\gamma}f(x) ,
\end{eqnarray*}
which completes the proof.
\end{proof}

%%%%%%%%%%%%%%%%%%%%%%%%%%%%%%%%%%%%%%%%%%%%%%%%%%%%%%%%%%%%%%%%%%%%%%%%%%%%%%%%%%%%%%%%%%%%%

\section{Appendix B}
From Proposition \ref{sp}, the weak-type boundedness of $M_{T,s}$ is a sufficient condition to have the sparse domination. It is natural to ask whether such condition be a necessary condition or not. However, it seems that the answer of this question is negative from following observations.
\begin{pro}
\noindent$({\rm{i}})$ Let $1 \le r  <\infty$. Then, there exist $f \in L^{\infty}_{c}$ and collction of sparse families $\{ \Sp(Q)\}_{Q:cube}$, and measurable set $K$ which has a non-zero measure, such that
\begin{eqnarray*}
\sup_{Q \in x } {\|\Lambda_{\Sp(Q),r}(f1_{{(3Q)}^c }  ) \|}_{L^{\infty}(Q)} =\infty
\end{eqnarray*}
for any $x \in K$.

\noindent$({\rm{ii}})$ Let $1 \le r  <s \le\infty$. Then, there exist $f \in  L^{\infty}_{c}$ and collction of sparse families $\{ \Sp(Q)\}_{Q:cube}$, and measurable set $K$ which has a non-zero measure, such that
\begin{eqnarray*}
\sup_{Q \ni x} \sup_{  {\| g \|}_{L^{s'}(Q) }=1}{|Q|}^{1/s} \Lambda_{\Sp(Q), r ,s'} (f1_{\Rn \setminus 3Q} ,g) =\infty
\end{eqnarray*}
for any $x \in K$.
 
\end{pro}
\begin{proof}
\noindent$({\rm{i}})$ Let fix a cube $Q_0$ and let $f=1_{Q_0}$. Furthermore, we define sparse collection $\Sp(Q)$ for any cube $Q$ by
\begin{eqnarray*}
\Sp(Q)=\{ 3^kQ \  ; \  k =1,2,3, \cdots \}.
\end{eqnarray*}
For any cube $3Q \subset Q_0$ and $z \in Q$, we choose $N \in \N$ such that $3^{N+1}Q \cap {Q}^{c}_{0} \neq \emptyset $ and $3^N  Q \subset Q_0$. Then, we have
\begin{eqnarray*}
\Lambda_{\Sp(Q),r}(f1_{\Rn\setminus  3Q}  ) (z)  1_Q(z)&=&  \sum^{\infty}_{ k=1 }  {\lan 1_{Q_0 \setminus 3Q} \ran}_{r, 3^kQ} 1_{Q}(z) \\
&\ge& \sum^{N}_{ k=1} {\left( \frac{|3^kQ \setminus 3Q| }{|3^kQ|}  \right)}^{1/r} 1_{Q}(z) \\
&\gtrsim& N1_Q(z),
\end{eqnarray*}
which yields
\begin{eqnarray*}
{\|\Lambda_{\Sp(Q),r}(f1_{\Rn \setminus  3Q}  ) \|}_{L^{\infty}(Q)} \gtrsim N.
\end{eqnarray*}
Since $N \to \infty $ at $|Q| \to 0$, we have
\begin{eqnarray*}
\sup_{Q \in x } {\|\Lambda_{\Sp(Q),r}(f1_{\Rn \setminus  3Q}  ) \|}_{L^{\infty}(Q)} =\infty
\end{eqnarray*}
for any $x \in Q_0$.
 
  \ 
  
\noindent$({\rm{ii}})$ By taking $f$ and $\Sp(Q)$ as above, we have 
\begin{eqnarray*}
\sup_{  {\|g\|}_{L^{s'}(Q) }=1} \Lambda_{\Sp(Q), r ,s'} (f 1_{\Rn \setminus 3Q} ,g)&=&\sup_{  {\|g\|}_{L^{s'}(Q) }=1} \sum^{\infty}_{k=1} |3^kQ| {\lan 1_{Q_0 \setminus 3Q}  \ran}_{r,3^kQ} {\lan g \ran}_{s',3^kQ} \\
&\ge& \sum^{N}_{k=1} |3^kQ| {\left( \frac{|3^kQ \setminus 3Q|}{|3^kQ|}  \right)}^{1/r}  {|Q|}^{-1/s'}{\left( \frac{|Q| }{|3^kQ|}  \right)}^{1/s'} \\
&\gtrsim&{|Q|}^{1/s} 3^{Nn/s}, 
\end{eqnarray*}
which complete the proof.
\end{proof}

$\bf{Acknowledgement}$. The author would like to thank Professor Mitsuru Sugimoto for constructive comments.

%%%%%%%%%%%%%%%%%%%%%%%%%%%%%%%%%%%%%%%%%%%%%%%%%%%%%%%%%%%%%%%%%%%%%%%%%%%%%%%%%%%%%%%

\end{document}